\documentclass[11pt,reqno]{amsart}
\usepackage{fullpage,times,graphicx,amssymb,amsmath,xcolor}
\usepackage[numbers]{natbib}
\usepackage[colorlinks,citecolor=bbluegray,linkcolor=ddarkbrown,urlcolor=blue,breaklinks]{hyperref}
\usepackage{bbding}
\usepackage{algorithmic,algorithm}

\usepackage{comment}

\usepackage{enumitem}
\usepackage{amsmath} 
\usepackage{amsthm} 

\newtheorem{remark}{Remark}

\definecolor{ddarkbrown}{rgb}{0.5,0.2,0.05} \definecolor{bbluegray}{rgb}{0.05,0,0.5}

\newtheorem{theorem}{Theorem}[section]
\newtheorem{proposition}[theorem]{Proposition}

\newtheorem{definition}[theorem]{Definition}
\newtheorem{lemma}[theorem]{Lemma}
\newtheorem{corollary}[theorem]{Corollary}

\newcommand{\BEAS}{\begin{eqnarray*}}
\newcommand{\EEAS}{\end{eqnarray*}}
\newcommand{\BEA}{\begin{eqnarray}}
\newcommand{\EEA}{\end{eqnarray}}
\newcommand{\BEQ}{\begin{equation}}
\newcommand{\EEQ}{\end{equation}}
\newcommand{\BIT}{\begin{itemize}}
\newcommand{\EIT}{\end{itemize}}
\newcommand{\BNUM}{\begin{enumerate}}
\newcommand{\ENUM}{\end{enumerate}}

\newcommand{\BA}{\begin{array}}
\newcommand{\EA}{\end{array}}

\let \oldsection \section
\renewcommand{\section}{\vspace{3ex plus 1ex}\oldsection}

\makeatletter
\def\@setdate{\@date}
\makeatother

\begin{document}

\title{Local and Global Uniform Convexity Conditions}

\author{Thomas Kerdreux $^{\dagger,\ast}$}
\address{Zuse Institute Berlin \& Technische Universit\"at Berlin, Germany}
\email{thomaskerdreux@gmail.com}

\author{Alexandre d'Aspremont$^{\ddagger,\mathsection}$}
\address{CNRS \& D.I., UMR 8548,\vskip 0ex
\'Ecole Normale Sup\'erieure, Paris, France.}
\email{aspremon@ens.fr}

\author{Sebastian Pokutta$^{\dagger,\ast}$}
\address{Zuse Institute Berlin \& Technische Universit\"at Berlin, Germany}
\email{pokutta@zib.de}



\keywords{}

\date{$^\dagger$Zuse Institute, Berlin, Germany.\\
\indent$^\ast$Technische Universit{\"a}t, Berlin, Germany.\\
\indent$^\ddagger$CNRS UMR 8548.\\
\indent$^\mathsection$D.I. \'Ecole Normale Sup\'erieure, Paris, France.}
\subjclass[2010]{}

\maketitle

\begin{abstract}
We review various characterizations of uniform convexity and smoothness on norm balls in finite-dimensional spaces and connect results stemming from the geometry of Banach spaces with \textit{scaling inequalities} used in analyzing the convergence of optimization methods. In particular, we establish local versions of these conditions to provide sharper insights on a recent body of complexity results in learning theory, online learning, or offline optimization, which rely on the strong convexity of the feasible set. While they have a significant impact on complexity, these strong convexity or uniform convexity properties of feasible sets are not exploited as thoroughly as their functional counterparts, and this work is an effort to correct this imbalance. We conclude with some practical examples in optimization and machine learning where leveraging these conditions and localized assumptions lead to new complexity results.
\end{abstract}

\section{Introduction}
Strong convexity or uniform convexity properties of the objective function of an optimization problem have a significant impact on problem complexity \citep{Nest15} and are heavily exploited by first order methods, notably in machine learning, with applications in various settings such as distributed optimization \citep{jaggi2014communication,lee2015distributed,ma2015adding,smith2017cocoa,stich2018local}, differential privacy \citep{talwar2014private,zhang2017efficient,chen2019renyi,iyengar2019towards,bassily2019private,kuru2020differentially,feldman2020private}, game theory \citep{du2019linear,liang2019interaction,abernethy2019last,mokhtari2020unified}.

While the impact of strong convexity or uniform convexity of the objective function is well understood. That of similar conditions on the feasible set of optimization problems is a priori just as significant but has been much less explored.
Despite the recent growing literature leveraging such set structure, which we now briefly survey, equivalent characterizations of strong convexity of sets and related weaker conditions are only sparsely covered.
This is arguably leading to some confusion, \textit{e.g.}, the notion of gauge sets introduced in \cite{abernethy2018faster} is equivalent to strong convexity \citep{Molinaro20}.
Another key motivation of our work is that, to our knowledge, only two results \citep{dunn1979rates,kerdreux2020uc} consider \emph{local} strongly convex assumptions of the constraint set to describe \emph{global} machine learning problem complexity, while these local properties have a significant impact on algorithm performance. 
This is surprising given the vast amount of literature around localized properties of objective functions, such as Kurdyka-{\L}ojasiewicz properties \citep{bolte2007lojasiewicz} for instance, leveraged in the convergence analyses of first-order optimization methods \citep{Bolte10,Attouch10,bolte2017error,roulet2020sharpness,kerdreux2019restarting,kerdreux2020accelerating}.

Uniform convexity (UC) generalizes strong convexity to more precisely quantify the curvature of a convex set, and plays a central role in many fields. For instance, the geometry of a Banach space is greatly influenced by its unit ball's uniform convexity, which notably drives the convergence behavior of martingales, and induces several concentration inequalities \citep{pisier1975martingales,pinelis1994optimum,iouditski2014primal}.

\subsection*{Gauges}
For simplicity, we focus here on compact convex sets $\mathcal{C}$ in finite-dimensional spaces. The gauge function of $\mathcal{C}$ provides a correspondence between sets and norm-like functions \citep{rockafellar1970convex} and is defined as
\begin{equation}\label{eq:gauge_function}\tag{Gauge}
\|x\|_{\mathcal{C}} \triangleq \text{inf}\big\{ \lambda \geq 0 ~|~ x\in\lambda\mathcal{C}\big\}.
\end{equation}
For simplicity again, we will only consider centrally symmetric convex bodies with nonempty interior in what follows, whose gauge function induces then a norm.

\subsection*{Uniformly Convex Sets in Optimization}
Some feasible set structures lead to accelerated convergence rates for first-order algorithms, \textit{e.g.}, projection-free algorithms.
Conditional gradients, a.k.a. Frank-Wolfe (FW) algorithms, are known to enjoy accelerated convergence rates compared to the $\mathcal{O}(1/T)$ baseline when the set is globally strongly convex \citep{polyak1966existence,demyanov1970,garber2015faster}.
However, to our knowledge, only two results in machine learning consider local strong convexity assumptions on the feasible set. 
\cite{dunn1979rates} proposes a geometrical condition on a given point $x^*\in\partial\mathcal{C}$ ensuring accelerated convergence rates for Frank-Wolfe algorithms and \cite{kerdreux2020uc} then show that this assumption is equivalent to local strong convexity and further generalizes all existing accelerated Frank-Wolfe regimes to hold also on locally uniformly convex sets.

Other projection-free algorithms exist with improved guarantees on strongly convex sets, \textit{e.g.}, for non-convex optimization \citep{rector2019revisiting}, min-max problems \citep{gidel2017frank,wang2018acceleration} or approximate Carathéodory results \citep{combettes2019revisiting}. The various equivalent definitions of strongly convex sets have also stimulated an interest in designing and analysing affine-invariant first-order methods. For instance, \cite{d2018optimal} proposed a choice of norm and prox-function in the implementation of first-order accelerated methods from \citep{nesterov2005smooth} which make these methods affine-invariant and provably optimal for optimization problems constrained on uniformly convex $\ell_p$ balls with $p>1$. \cite{kerdreux2020aff} proposed an optimal (w.r.t. known analyses) affine-invariant analysis of the affine-covariant Frank-Wolfe algorithm on strongly convex sets. Their analysis rely on assumptions that combine scaling inequalities for strongly convex feasible sets and an affine-invariant characterization of smoothness \citep{jaggi2013revisiting}. Finally, strong convexity for sets was also used outside of projection-free optimization techniques in, \textit{e.g.}, \citep{veliov2020gradient,bach2020effectiveness}.

\subsection*{Uniformly Convex Sets in Machine Learning.}
The global strong convexity of sets also characterizes performance in learning theory and online learning.
\cite{huang2016following,huang2017following} studied logarithmic regret bounds of simple algorithms for online linear learning on smooth strongly convex decisions sets. \cite{Molinaro20,kerdreux2020uc} later extended these results to non-smooth and uniformly convex sets. \citep{abbasi2009forced,rusmevichientong2010linearly} considered such assumptions of the constraint set for stochastic linear bandits and \citep{abernethy2009beating,bubeck2018sparsity} for non-stochastic linear bandits.
The global uniform convexity of the decision set has recently attracted much attention in ``online learning with a hint'', which is a multiplicative version of optimistic online learning. In this framework, regret bounds are obtained in terms of the uniform convexity power type of the decision set \citep{dekel2017online,bhaskara2020online,bhaskara2020onlineMany}.

\cite{kakade2009complexity} studied generalization bounds of low-norm linear classes. They obtain upper bounds on the Rademacher constant of the hypothesis class that depend on the strong convexity of the norm regularizing the class. However, they expressed these results in terms of the functional strong convexity of the square of the norm. In Section \ref{ssec:kakade}, we recall that this result is a quantitative corollary of known results in the geometrical study of Banach spaces: a uniformly convex space has a non-trivial Rademacher type. \cite{el2019generalization} also consider global strong convexity of the feasible region to strengthen convergence results in generalization bounds in the \emph{Predict-Then-Optimize} framework. They notably rely on a characterization of strong convexity akin to \emph{scaling inequalities} covered in \ref{itm:scaling_inequality} of Theorems \ref{th:UC_set_and_support}-\ref{thm:local_UC_smoothness_gauge}.

In online learning on Banach spaces, several works analyse regret bounds in terms of the martingale type/cotype of the space \citep{sridharan2010convex,srebro2011universality}, a property directly tied with uniform convexity. 
In fact, \citep{sridharan2010convex,srebro2011universality} relies on the fact that the martingale type of a space is related to the existence of a uniformly convex function on this space, see \citep[Theorem 1]{sridharan2010convex}. Besides, as we recall in Section \ref{ssec:kakade}, a uniformly convex space has also a Rademacher type (the reverse might not be true), a notion related to the martingale type. This martingale type structure has been leveraged in various applications in learning \citep{Schn16,kerdreux2017approximate} as it is a central tool to derive concentration inequalities \citep{pisier1975martingales,pinelis1994optimum,pisier2011martingales}. 
However, our main focus here remains on uniform convexity as it has a simple geometrical interpretation in terms of scaling inequalities with direct algorithmic consequences (items \ref{itm:scaling_inequality} in Theorems \ref{th:UC_set_and_support}-\ref{thm:local_UC_smoothness_gauge}), and admits local versions (Theorem \ref{thm:local_UC_smoothness_gauge}) which also better characterize empirical performance, as opposed to martingale type/cotype properties.

\subsection*{Contributions}
We first provide elementary proofs of various local and global equivalent characterizations of uniform convexity of sets. We then discuss applications in machine learning and cover some practical examples leveraging these alternative points of view in Section \ref{sec:applications}. Most of our results are quantitative. 

We then characterize the uniform convexity of a set in terms of the ``angles" between normal cone directions and feasible directions at boundary points. These quantifications appear regularly in convergence proofs of algorithms such as Frank-Wolfe and we call them \emph{scaling inequalities}. The link with uniform convexity is often ignored and our objective here is to explicitly quantify this connection.

Finally, we derive equivalent relationships for the localized versions of UC (see Theorem \ref{thm:local_UC_smoothness_gauge}) to better explain empirical performance in optimization methods.

\subsection*{Related Works.}
Our work connects different perspectives of uniform convexity of a set. Our Theorems \ref{th:UC_set_and_support}-\ref{thm:local_UC_smoothness_gauge} rely on several classical monographs. We refer to \citep{zalinescu1983uniformly,aze1995uniformly,zalinescu2002convex} for the study of functional uniform convexity and smoothness, to \citep{lindenstrauss2013classical,beauzamy2011introduction,deville1993smoothness,borwein2009uniformly} for the study of the geometry of Banach spaces in terms of uniform convexity and smoothness, and to \citep{pisier2011martingales} for results on type/cotype properties of a Banach space. We also invoke \citep{goncharov2017strong} for practical local characterizations of the strong convexity of sets.
Finally, we rely on \citep{rockafellar1970convex} for convex analysis references and on \citep{schneider2014convex} for convex geometry in finite dimensions. Whenever possible, we keep track of the precise reference to these monographs when establishing the results in Sections \ref{sec:3_point_of_view}-\ref{sec:local_equivalence}. In many cases, we have adapted the proofs to make the results quantitative.

\subsection*{Outline.}
In Section \ref{sec:preliminaries} we group some preliminary facts and in Section \ref{sec:3_point_of_view}, we recall the definition of uniform convexity and smoothness for functions and spaces. 
In Section \ref{sec:global_equivalence}, we present Theorem \ref{th:UC_set_and_support} stating different equivalent definitions of the uniform convexity of a norm ball in finite-dimensional spaces.
Theorem \ref{thm:local_UC_smoothness_gauge} in Section \ref{sec:local_equivalence} provides the same results but with local assumptions.
Results in Section \ref{sec:global_equivalence}-\ref{sec:local_equivalence} are self-contained and proofs are elementary. 
However, they hold even in infinite-dimensional spaces. 
Finally, in Section \ref{sec:applications}, we provide three examples in offline optimization and learning theory where these different points of view on uniform convexity lead to new results.

\subsection*{Notations}
The finite-dimensional ambient vector space is $\mathbb{R}^m$ and by $\text{Int}(\mathcal{C})$ and $\partial\mathcal{C}$, we denote the interior of $\mathcal{C}$ and the boundary of $\mathcal{C}$ respectively.
The \emph{support function} of $\mathcal{C}$ is defined as $\sigma_{\mathcal{C}}(d)\triangleq \text{sup}_{v\in\mathcal{C}} \langle v; d \rangle$.
The \emph{normal cone} of $\mathcal{C}$ at $x^*\in\mathcal{C}$ is defined as $N_{\mathcal{C}}(x^*) \triangleq \big\{ d ~| ~ \langle d; x - x^* \rangle \leq 0 ~~\forall x\in\mathcal{C} \big\}$ and the support set of $\mathcal{C}$ at $d$ is $F_{\mathcal{C}}(d)\triangleq\big\{x\in\mathcal{C}~|~\langle x;d\rangle=\sigma_{\mathcal{C}}(d)\big\}$.
We write $f^*(y)={\text{sup}}_{x\in\mathbb{R}^m} \langle x ; y\rangle - f(x)$ as the \emph{Fenchel conjugate} of $f$.
We will consider convex functions $f:\mathbb{R}^m\rightarrow \mathbb{R}$, finite everywhere and continuous. 
In particular, we then have that $f^{**}=f$.
For a norm $\|\cdot\|$, we write $\|x\|_\star \triangleq \text{sup }\big\{ \langle x; y\rangle ~|~ \|y\|\leq 1\big\}$ to denote its dual norm.
We sometimes also use $\|\cdot\|^\star$.
We use different star symbols to distinguish between dual norms and Fenchel dual, \textit{e.g.}, the Fenchel dual of a norm is not the dual norm in general.
We write $B_{\|\cdot\|}$ the unit ball and $S_{\|\cdot\|}$ the unit sphere associated to a norm $\|\cdot\|$.
We most often consider $(p,q)$ s.t. $p\geq 2$, $q\in]1,2]$ and $1/p+1/q=1$. The $p$ (resp. $q$) parameter will hence be employed in the context of uniform convexity (resp. smoothness).

\section{Preliminaries}\label{sec:preliminaries}
We restrict the discussion to finite-dimensional spaces for simplicity. 
It allows for a direct analogy of duality between a norm and its dual norm with the duality between the norm ball's gauge function and the support function of the norm ball's polar, which we detail now.
Note that results similar to Theorems \ref{th:UC_set_and_support} and \ref{thm:local_UC_smoothness_gauge} hold in infinite-dimensional Banach spaces though.
We consider centrally symmetric convex bodies $\mathcal{C}$ with non-empty interior
so that the gauge function $\|\cdot\|_{\mathcal{C}}$ of $\mathcal{C}$ is a norm \citep[Theorem 15.2.]{rockafellar1970convex}.
In particular, the unit ball (resp. the sphere) of $\|\cdot\|_{\mathcal{C}}$ corresponds to $\mathcal{C}$ (resp. $\partial\mathcal{C}$), \textit{i.e.}, $\mathcal{C}=B_{\|\cdot\|_{\mathcal{C}}}$ and $\partial\mathcal{C}=S_{\|\cdot\|_{\mathcal{C}}}$. The function $\|\cdot\|_{\mathcal{C}}$ and $\sigma_{\mathcal{C}}$ are every-where finite convex functions from $\mathbb{R}^m$ to $\mathbb{R}_+$ and, \textit{e.g.}, subdifferentiable \citep[Theorem 23.4]{rockafellar1970convex}.

A strictly convex set $\mathcal{C}$ is such that for any distinct $(x,y)\in\partial\mathcal{C}$, we have $(x+y)/2\in\mathcal{C}\setminus\partial\mathcal{C}$.
Conversely, $\mathcal{C}$ is smooth if there is only one supporting hyperplane at each boundary point of $\mathcal{C}$.
The following lemma recalls the classical relation between strict convexity of a set and differentiability of the support function \citep[Cor 1.7.3]{schneider2014convex}.

\begin{lemma}[Support/Gauge Differentiability]\label{lem:differentiability_support}
Consider $\mathcal{C}\subset\mathbb{R}^m$ a compact convex set.
$\sigma_{\mathcal{C}}$ is differentiable at $d\in\mathbb{R}^m\setminus\{0\}$ if and only if $\{y~|~\langle y; d\rangle = \sigma_{\mathcal{C}}(d)\}=\{x\}$. In that case $\nabla\sigma_{\mathcal{C}}(d)=x$.
In particular, if $\mathcal{C}$ is strictly convex, then $\sigma_{\mathcal{C}}$ is differentiable on $\mathbb{R}^m\setminus\{0\}$.
\end{lemma}
The \emph{polar} of $\mathcal{C}$ is defined as $\mathcal{C}^\circ = \big\{ d\in \mathbb{R}^m~|~ \langle x; d\rangle \leq 1 ~~\forall x\in\mathcal{C} \big\}$. 
Importantly, the support and gauge function are dual to each other via the polar operation, \textit{i.e.},  $\sigma_{\mathcal{C}}(\cdot)=\|\cdot\|_{\mathcal{C}^\circ}$ \citep[Theorem 14.5.]{rockafellar1970convex}.
We systematically write $x$ (resp. $d$) for an element of $\mathcal{C}$ (resp. $\mathcal{C}^\circ$).
This duality parallels that of a norm and its dual.  Indeed, if $\|\cdot\|_{\mathcal{C}}$ is a norm, then $\|\cdot\|_{\mathcal{C}^\circ}$ is a norm and $\|\cdot\|_{\mathcal{C}}^\star = \|\cdot\|_{\mathcal{C}^\circ}$ \citep[Cor 15.1.2]{rockafellar1970convex}.
Finally, the following classical lemma will be particularly useful \citep[Lemma 2]{asplund1968frechet}.
\begin{lemma}\label{lem:power_norm_fenchel}
Let $p,q >1$ s.t. $\frac{1}{p} + \frac{1}{q}=1$. 
Then, for any $\alpha>0$, we have
\[
\big(\alpha\sigma_{\mathcal{C}}^p\big)^*(\cdot) = \Big[\frac{1}{(\alpha p)^{1/(p-1)}} - \frac{\alpha}{(\alpha p)^{q}}\Big]\|\cdot\|_{\mathcal{C}}^{q}.
\]
In particular for $\alpha=\frac{1}{p}$, it means that
the Fenchel conjugate of $\frac{1}{p}\|\cdot\|^p$ is $\frac{1}{q}\|\cdot\|^q_\star$.
\end{lemma}
\begin{proof}[Proof of Lemma \ref{lem:power_norm_fenchel}]
We recall the proof for completeness. Consider $\rho^*(u)\triangleq \text{sup}_{t>0}\big(tu - \rho(t)\big)$. For any $y$, we have
\begin{eqnarray*}
\rho^*(\|y\|_\star) & = & \text{sup}_{t>0} \big\{ t\|y\|_{\star} - \rho(t) \big\} =  \text{sup}_{t>0} \text{sup}_{x \neq 0} \Big[ t \langle y; x\rangle/ \|x\| - \rho(t) \Big]\\
& = & \text{sup}_{t>0; x\neq 0} \Big[ t\frac{\langle y; x t/\|x\|\rangle}{\|x t/\|x\|\|} - \rho(t) \Big] = \text{sup}_{t>0; x\neq 0} \Big\{ t\frac{\langle y; x\rangle}{\|x\|} - \rho(t) ; \|x\|=t\Big\}\\
& = & \text{sup}_{x\neq 0} \big\{ \langle y; x\rangle - \rho(\|x\|)\big\} = (\rho\circ\|\cdot\|)^*(y).
\end{eqnarray*}
Also, an immediate calculation proves that for $u\geq 0$ and when $\rho(t) = \alpha t^r$ with $r>1$, we have $\rho^*(u) = \Big[\frac{1}{(\alpha r)^{1/(r-1)}} - \frac{\alpha}{(\alpha r)^{r/(r-1)}}\Big] u^{r/(r-1)}$. 
We finally conclude noting that $\sigma_{\mathcal{C}^\circ}(\cdot)=\|\cdot\|_{\mathcal{C}}$ and $\|\cdot\|_{\mathcal{C}^\circ}=\|\cdot\|_\star$.
\end{proof}

\section{Spaces, Sets, Functions Uniform Smoothness and Convexity}\label{sec:3_point_of_view}

In this section, we introduce the necessary concepts to state the main theorems in Sections \ref{sec:global_equivalence}-\ref{sec:local_equivalence}.
We recall the classical notions of uniform convexity and smoothness for functions (Section \ref{ssec:functional_UC_UC}) and Banach spaces (Section \ref{ssec:set_space_UC_UC}).
We also recall quantitative statements on the duality correspondence between smoothness and uniform convexity in each of these situations.

\subsection{Uniform Convexity and Smoothness of Functions}
\label{ssec:functional_UC_UC}
Uniform convexity and smoothness of functions were introduced to analyse optimization algorithms \citep{polyak1966existence} and extensively studied in
\citep{zalinescu1983uniformly,aze1995uniformly,zalinescu2002convex}, and is now a standard assumption in the analysis of first order methods, see, \textit{e.g.}, \citep{iouditski2014primal}.

The following equivalent definitions of uniformly smooth function are classical, see, \textit{e.g.}, \citep[(i)-(iv)-(ix) of Theorem 3.5.6.]{zalinescu2002convex}, which notably shows that a continuous uniformly smooth function is Fréchet differentiable. This means that a norm for instance is not uniformly smooth as it is not differentiable at $0$, see Lemma \ref{lem:differentiability_support}.
This explains why hypothesis \ref{itm:modulus_UC} in Theorem \ref{th:UC_set_and_support} below is restricted to $S_{\|\cdot\|}(1)$. In the following sections, we consider only uniform convexity and smoothness of functions to ultimately apply it to simple transformations of the gauge and support functions. 
We recall self-contained proofs of the equivalences in the definition to obtain quantitative statements. Note that whenever we invoke uniformly smooth or convex functions in the other sections, we will often refer to these \emph{zero-order characterization}.

\begin{definition}[Uniformly Smooth Functions]\label{def:uniformly_smooth_functions}
Consider a convex function $f:\mathbb{R}^m\rightarrow\mathbb{R}$ and $q\in]1,2]$.
The following assertions are equivalent
\begin{enumerate}[label=(\alph*)]
    \item (Zero-order) There exists $c>0$ s.t. $f$ is $(c, q)$-uniformly smooth with respect to $\|\cdot\|$, \textit{i.e.}, for any $(x, y)$ and $\lambda\in[0,1]$
    \[
    f(\lambda x + (1-\lambda)y) + (c/q) \lambda (1-\lambda) \|x-y\|^q \geq \lambda f(x) + (1-\lambda) f(y).
    \]
    \label{itm:fct_zero_order_smoothness}

    \item (First-order) $f$ is differentiable and there exists $c^\prime>0$ such that for any $(x,y)$, we have
    \[
    f(y) \leq f(x) + \langle \nabla f(x); y - x \rangle + \frac{c^\prime}{q} \|x-y\|^q.
    \]
    \label{itm:fct_first_order_smoothness}

    \item (H\"older gradient) $f$ is differentiable and there exists $c^{\prime\prime}>0$ such that $f$ is $(c^{\prime\prime}, q)$-H\"older-smooth w.r.t. $\|\cdot\|$, \textit{i.e.}, for any $(x,y)$
    \[
    \big\| \nabla f(x) - \nabla f(y) \big\|_\star \leq c^{\prime\prime} \| x - y\|^{q-1}.
    \]
    \label{itm:fct_holder_gradient_smoothness}
\end{enumerate}
\end{definition}
\begin{proof}[Proof of equivalency in Definition \ref{def:uniformly_smooth_functions}]
We adapt the proof of \citep[Theorem 3.5.6]{zalinescu2002convex} to our case.

\ref{itm:fct_zero_order_smoothness} $\implies$ \ref{itm:fct_first_order_smoothness}.
Let $(x,y)\in\mathbb{R}^m$ and $\lambda\in ]0,1]$. The zero-order condition evaluated at $(x,y)$ implies that
\begin{equation}\label{eq:zero_oder_smooth_divided_lambda}
\frac{f(y + \lambda (x-y)) - f(y)}{\lambda} + (c/q)(1 - \lambda) \|x - y\|^p \geq f(x) - f(y).
\end{equation}
And because $f$ is a finite convex function, the limit of $\big(f(y + \lambda (x-y)) - f(y)\big)/\lambda$ when $\lambda$ converges to $0^+$ exists \citep[Theorem 23.1.]{rockafellar1970convex} and $f^\prime(x,\cdot)$ is defined for $d\in\mathbb{R}^m$ as
\[
f^\prime(x, d) \triangleq \underset{\lambda\rightarrow 0^+}{\text{lim }} \frac{f(y+\lambda d) - f(y)}{\lambda}.
\]
In particular, with $d=x-y$, it implies in \eqref{eq:zero_oder_smooth_divided_lambda} that
\begin{equation}\label{eq:first_order_not_yet}
f^\prime(y, x - y) + (c/q) \|x - y\|^q \geq f(x) - f(y).
\end{equation}
Let us now show that $f^\prime(x, \cdot)$ is linear.
By definition of $f^\prime(x, \cdot)$, we have that $f^\prime(x, y) \geq -f^\prime(x,-y)$. 
Let us now show that the other side inequality is also true.
Summing the two versions of \eqref{eq:first_order_not_yet} by interchanging $x$ and $y$, we obtain
\begin{equation*}
f^\prime(y, x-y) + f^\prime(x, y-x) + (2c/q) \|x - y\|^q \geq 0.
\end{equation*}
Let $u\in\mathbb{R}^m$, $t>0$ and write $\rho(t) = f(x + tu)$. Then
\begin{equation*}
  \left\{
    \begin{split}
    \rho_+^\prime(t) &\triangleq \underset{\lambda\rightarrow t^+}{\text{lim }}\frac{\rho(\lambda) - \rho(t)}{\lambda - t} = f^\prime(x + t u,u)\\
    \rho_{-}^\prime(t) &\triangleq \underset{\lambda\rightarrow t^-}{\text{lim }}\frac{\rho(\lambda) - \rho(t)}{\lambda - t} = -f^\prime(x + tu,-u).
    \end{split}
  \right.
\end{equation*}
Then, because $\rho$ is convex, we have $ \rho^\prime_{+}(-t)\leq \rho^\prime_-(0)\leq \rho^\prime_+(0) \leq \rho^\prime_-(t)$.
Hence, for any $t>0$
\[
f^\prime(x,u) + f^\prime(x, -u) = \rho^\prime_+(0) - \rho^\prime_-(0) \leq \rho^\prime_{-}(t) - \rho^\prime_{+}(-t) = - \big[ f^\prime(x+tu,-u) + f^\prime(x-tu, u)\big] \leq \frac{2c}{q}(2t)^q\|u\|^q.
\]
We conclude that $f^\prime(x,u)\leq -f^\prime(x,-u)$ and finally that $f^\prime(x,u) = -f^\prime(x,-u)$.
Hence $f^\prime(x,\cdot)$ is a bounded linear function for any $x$ so that $f$ is differentiable with $f^\prime(x,h)=\langle \nabla f(x); h\rangle$. We conclude by letting $\lambda$ converging to $1$ in  \eqref{eq:zero_oder_smooth_divided_lambda}.

\ref{itm:fct_first_order_smoothness} $\implies$ \ref{itm:fct_zero_order_smoothness}.
Write $x_{\lambda}=\lambda x + (1 - \lambda) y$. Applying the first order at $x = x_{\lambda} + x - x_{\lambda}$ and $y = x_{\lambda} + y - x_{\lambda}$, we obtain
\begin{eqnarray*}
f(x) & \leq & f(x_\lambda) + (1 - \lambda) \langle \nabla f(x_\lambda); x - y\rangle + (c/q)(1-\lambda)^q \|x - y\|^q\\
f(y) & \leq & f(x_\lambda) + \lambda \langle \nabla f(x_\lambda); y - x \rangle + (c/q)\lambda^q \|x - y\|^q.
\end{eqnarray*}
Then, by multiplying the inequalities respectively with $\lambda$ and $1-\lambda$ and summing then, we obtain
\[
\lambda f(x) + (1 - \lambda) f(y) \leq f(x_\lambda) + (c/q)\lambda(1-\lambda)\big[ (1-\lambda)^{q-1} + \lambda^{q-1} \big] \|x - y\|^q.
\]
Then, by symmetry of $(1-\lambda)^{q-1} + \lambda^{q-1}$ and because $q-1\in]0,1]$, we obtain
\[
\lambda f(x) + (1 - \lambda) f(y) \leq f(x_\lambda) + 2(c/q)\lambda(1-\lambda)\|x - y\|^q.
\]

\ref{itm:fct_first_order_smoothness} $\implies$ \ref{itm:fct_holder_gradient_smoothness}.
For any $z$, by convexity of $f$, we have $f(y+z) \geq f(x) + \langle \nabla f(x); y+z-x \rangle$ and by assumption we have $f(y+z) \leq f(y) + \langle \nabla f(y) ;  z \rangle + \frac{c^{\prime}}{q}\|z\|^q$. Hence
\[
 f(y) + \langle \nabla f(y) ;  z \rangle + \frac{c^{\prime}}{q}\|z\|^q \geq f(x) + \langle \nabla f(x); y+z-x \rangle
\]
so that for any $z$
\[
\langle z; \nabla f(x) - \nabla f(y) \rangle - \frac{c^\prime}{q} \|z\|^q \leq f(y) - f(x) + \langle \nabla f(x); x-y\rangle \leq \frac{c^\prime}{q} \|x-y\|^q.
\]
Then, by taking the supremum over $z$ on both sides, we obtain
\[
c^\prime \Big(\frac{1}{q}\|\cdot\|^q\Big)^*\big((\nabla f(x)-\nabla f(y))/c^\prime\big) \leq  \frac{c^\prime}{q} \|x-y\|^q.
\]
With $p\geq 2$ s.t. $\frac{1}{p}+\frac{1}{q}=1$, Lemma \ref{lem:power_norm_fenchel} then implies 
\[
\frac{c^\prime}{p}\Big\|\frac{\nabla f(x) - \nabla f(y)}{c^\prime}\Big\|_\star^p \leq \frac{c^\prime}{q} \|x-y\|^q.
\]
In particular, $\frac{p}{q}=\frac{1}{q-1}$ and we obtain
\[
\|\nabla f(x) - \nabla f(y) \|_\star \leq \frac{c^{\prime\prime}}{(q-1)^{1/p}} \|x - y\|^{q-1}.
\]

\ref{itm:fct_holder_gradient_smoothness} $\implies$ \ref{itm:fct_first_order_smoothness}. By convexity of $f$, we have $f(x) \geq f(y) + \langle \nabla f(y); x-y\rangle$. Hence by definition of the dual norm, we obtain
\[
f(y) - f(x) - \langle \nabla f(x); y-x\rangle \leq \langle \nabla f(y) - \nabla f(x); y - x\rangle \leq \|\nabla f(y) - \nabla f(x)\|_\star \|y-x\|.
\]
and using the H\"older-smoothness of $f$ we obtain
\begin{equation*}\tag*{\qedhere} 
f(y) - f(x) - \langle \nabla f(x); y-x\rangle \leq  c^\prime\|y-x\|^q.
\end{equation*}
\end{proof}

We now define uniform convexity of a function, see, \textit{e.g.}, \citep[Definition 1]{aze1995uniformly}. We state the results in terms of subgradients as gauge or support functions are not necessarily differentiable.

\begin{definition}[Uniformly Convex Functions]\label{def:uniformly_convex_functions}
Consider a convex function $f:\mathbb{R}^m\rightarrow\mathbb{R}$ and $p\geq 2$.
The following assertions are equivalent
\begin{enumerate}[label=(\alph*)]
    \item (Zero-order) There exists $c>0$ s.t. $f$ is $(c, p)$-uniformly convex with respect to $\|\cdot\|$, \textit{i.e.}, for any $(x, y)$ and $\lambda\in[0,1]$, we have 
    \begin{equation*}
    f(\lambda x + (1-\lambda)y) + (c/p) \lambda (1-\lambda) \|x-y\|^p \leq \lambda f(x) + (1-\lambda) f(y).
    \end{equation*}
    \label{itm:fct_zero_order_convexity}

    \item (First-order) There exists $\alpha>0$ s.t. for any $(x,y)\in\mathcal{C}$ and $d\in\partial f(x)$, we have
    \begin{equation*}
    f(y) \geq f(x) + \langle d; y - x \rangle + \frac{\alpha}{p} \|x-y\|^p.
    \end{equation*}
    \label{itm:fct_first_order_convexity}
\end{enumerate}
\end{definition}

\begin{proof}[Proof of equivalency in Definition \ref{def:uniformly_convex_functions}]
\ref{itm:fct_zero_order_convexity} $\implies$ \ref{itm:fct_first_order_convexity}. 
Let $(x,y)\in\mathbb{R}^m$ and $d\in\partial f(y)$. Combining convexity of $f$ and zero-order uniform convexity, we have
\[
f(y) + \lambda\langle d; x-y \rangle \leq f(y + \lambda(x-y)) \leq  f(y) + \lambda (f(x) - f(y)) - (c/p) \lambda (1-\lambda)\|x-y\|^p.
\]
Then, dividing by $\lambda$ and evaluating with $\lambda$ converging to zero, we have
\[
\langle d; x-y \rangle \leq f(x) - f(y) - (c/p) \|x-y\|^p.
\]

\ref{itm:fct_first_order_convexity} $\implies$ \ref{itm:fct_zero_order_convexity}. Write $x_\lambda = \lambda x + (1 - \lambda)y$. We apply the first-order condition at $x = x_{\lambda} + x - x_{\lambda}$ and $y = x_{\lambda} + y - x_{\lambda}$. With $d\in\partial f(x_\lambda)$, we have
\begin{eqnarray*}
f(x) & \geq & f(x_\lambda) + (1- \lambda)\langle d; x-y\rangle + \frac{\alpha}{p} (1-\lambda)^p\|y-x\|^p\\
f(y) & \geq & f(x_\lambda) + \lambda\langle d; y - x\rangle + \frac{\alpha}{p} \lambda^p\|y-x\|^p.
\end{eqnarray*}
Multiplying the inequalities respectively by $\lambda$ and $1-\lambda$ and summing them, we obtain
\[
\lambda f(x) + (1-\lambda) f(y)\geq f(x_\lambda) + \frac{\alpha}{p}\lambda (1-\lambda) \big[(1-\lambda)^{p-1} + \lambda^p \big] \|y-x\|^p.
\]
Then, by symmetry, we have that $\text{min}_{\lambda\in[0,1]}\big[(1-\lambda)^{p-1} + \lambda^p \big]= 1/2^{p-2}$, which concludes that
\begin{equation*}\tag*{\qedhere}
\lambda f(x) + (1-\lambda) f(y)\geq f(x_\lambda) + \frac{\alpha}{2^{p-2}p}\lambda (1-\lambda)\|y-x\|^p.
\end{equation*}
\end{proof}

Uniform smoothness (US) and uniform convexity (UC) are dual properties by Fenchel conjugacy \citep[Theorem 2.1.]{zalinescu1983uniformly} or \citep[Proposition 2.6]{aze1995uniformly}.
We recall a proof below, both for completeness and to obtain quantitative statements.

\begin{proposition}[Uniform Smoothness and Convexity with Fenchel duality]\label{prop:Duality_Fenchel_UC_US}
Consider $\alpha,c>0$, $p\geq 2$ and $q\in]1,2]$ such that $\frac{1}{p}+\frac{1}{q}=1$, and a norm $\|\cdot\|$ with its dual norm $\|\cdot\|_\star$.
Let $f:\mathbb{R}^m\rightarrow\mathbb{R}$ be a convex function.
We have the following implications
\begin{enumerate}[label=(\alph*)]
    \item If $f$ is $(\alpha, q)$-uniformly smooth w.r.t. $\|\cdot\|$ (Definition \ref{def:uniformly_smooth_functions} \ref{itm:fct_zero_order_smoothness}), then $f^*$ is $(1/(p \alpha^{p-1}), p)$-uniformly convex w.r.t. $\|\cdot\|_\star$ (Definition \ref{def:uniformly_convex_functions} \ref{itm:fct_zero_order_convexity}). \label{itm:US_Fenchel_UC}
    
    \item If $f$ is $(c, p)$-uniformly convex w.r.t. $\|\cdot\|$, then $f^*$ is $\big(1/(q c^{q-1}), q\big)$-uniformly smooth with respect to $\|\cdot\|_\star$. \label{itm:UC_Fenchel_US}
\end{enumerate}
\end{proposition}

\begin{proof}[Proof of Proposition \ref{prop:Duality_Fenchel_UC_US}]
Let us prove \ref{itm:UC_Fenchel_US}, \ref{itm:US_Fenchel_UC} follows similarly.
Assume $f$ is $(c,p)$-uniformly convex. 
Consider $(y_1,y_2)\in \mathbb{R}^m$, $\lambda\in[0,1]$ and write $y_\lambda=\lambda y_1 + (1-\lambda) y_2$. 
Similarly, for any $(x_1,x_2)\in\mathbb{R}^m$, let us write $x_\lambda=\lambda x_1 + (1-\lambda)x_2$, $f(x_i)=f_i$ for $i=1,2$, $f(x_\lambda)=f_\lambda$, $f(x_\lambda)^*=f_\lambda^*$ etc.
By definition of conjugate functions, and using the zero-order uniform convexity of $f(\cdot)$ at $x_{\lambda}$, we have
\[
\langle y_{\lambda}; x_{\lambda}\rangle \leq f^*(y_{\lambda}) + f_\lambda\leq  f^*(y_{\lambda}) - (c/p) \lambda(1-\lambda)\|x_1 - x_2\|^p + \lambda f_1 + (1-\lambda)f_2.
\]
By adding and subtracting $\lambda(1-\lambda)\langle y_1 - y_2; x_1 - x_2\rangle$, we obtain
\[
\langle y_{\lambda}; x_{\lambda}\rangle \leq f^*(y_{\lambda}) + \lambda f_1 + (1-\lambda)f_2 - \lambda (1-\lambda) \langle y_1 - y_2; x_1 - x_2\rangle + \lambda(1-\lambda) \big[ \langle y_1 - y_2; x_1 - x_2\rangle  - (c/p) \|x_1 - x_2\|^p\big]. 
\]
The right term in brackets is upper bounded by $((c/p) \|\cdot\|^p)^*(y_1-y_2)$, so that
\[
\langle y_{\lambda}; x_{\lambda}\rangle  - \lambda f_1 - (1-\lambda)f_2 + \lambda (1-\lambda) \langle y_1 - y_2; x_1 - x_2\rangle \leq f^*(y_{\lambda}) + \lambda(1-\lambda) ((c/p) \|\cdot\|^p)^*(y_1-y_2).
\]
Note also the following equality
\[
\langle y_\lambda; x_\lambda\rangle + \lambda (1-\lambda) \langle y_1 - y_2; x_1 - x_2\rangle = \lambda \langle y_1;x_1\rangle + (1-\lambda)\langle y_2; x_2\rangle.
\]
Hence, we obtain
\begin{eqnarray*}
\lambda \langle y_1;x_1\rangle + (1-\lambda)\langle y_2; x_2\rangle  - \lambda f_1 - (1-\lambda)f_2 &\leq& f^*(y_{\lambda}) + \lambda(1-\lambda) ((c/p) \|\cdot\|^p)^*(y_1-y_2)\\
\lambda \big[\langle y_1;x_1\rangle - f_2\big]+ (1-\lambda)\big[\langle y_2; x_2\rangle  - f_2\big] &\leq& f^*(y_{\lambda}) + \lambda(1-\lambda) ((c/p) \|\cdot\|^p)^*(y_1-y_2).
\end{eqnarray*}
Because the last inequality is true for any $(x_1,x_2)$, we conclude that
\[
\lambda f^*(y_1) + (1-\lambda)f^*(y_2) \leq f^*(y_\lambda) + \lambda(1-\lambda) ((c/p) \|\cdot\|^p)^*(y_1-y_2).
\]
Lemma \ref{lem:power_norm_fenchel} implies that $((c/p) \|\cdot\|^p)^*(y_1-y_2) = \frac{1}{q c^{q-1}}\|y_1-y_2\|_\star^q$. Finally $f^*$ is $\big(1/(q c^{q-1}), q\big)$-uniformly smooth with respect to $\|\cdot\|_\star$.
\end{proof}

In the following proposition, we provide similar results for local notions of uniform convexity and smoothness of a function.
These are quantitative versions of \citep[Proposition 3.2.]{aze1995uniformly} or \citep[(iv) \& (v) Theorem 2.1.]{zalinescu1983uniformly}. 

\begin{proposition}\label{prop:local_directional_duality}
Consider $\alpha,c>0$, $p\geq 2$ and $q\in]1,2]$ such that $\frac{1}{p}+\frac{1}{q}=1$.
Let $f:\mathbb{R}^m\rightarrow\mathbb{R}$ a convex function and $(x,d)$ such that $d\in\partial f(x)$ and $x\in\partial f^*(d)$.
The following assertions are equivalent
\begin{enumerate}[label=(\alph*)]
    \item For some $\alpha>0$, $f^*$ is $(\alpha, q)$-uniformly-smooth at $d$ w.r.t. to $\|\cdot\|_\star$, \textit{i.e.}, for all $d_1$, we have
    \[
    f^*(d_1) \leq f^*(d) + \langle x; d_1-d\rangle + \frac{\alpha}{q}\|d_1-d\|_\star^q.
    \]
    \label{itm:local_directional_fct_smoothness}
    
    \item For some $c>0$, $f$ is $(c, p)$-uniformly convex at $x$ w.r.t $\|\cdot\|$, \text{i.e.}, for any $y$, we have
    \[
    f(y) \geq f(x) + \langle d; y-x\rangle + \frac{c}{p} \|y-x\|^{p}.
    \]
    \label{itm:local_directional_fct_convexity}

\end{enumerate}
\end{proposition}

\begin{proof}[Proof of Proposition \ref{prop:local_directional_duality}]
First note, that since $f$ is finite l.s.c., for $d\in\partial f(x)$, we have $x\in\partial f^*(d)$ \citep[Theorem 23.5.]{rockafellar1970convex}.
Let us show that \ref{itm:local_directional_fct_smoothness} $\implies$ \ref{itm:local_directional_fct_convexity}, the converse follows similarly. 
Recall that $f(y)=\text{sup}_{d_1\in\mathbb{R}^m}\big\{\langle y; d_1\rangle - f^*(d_1)\big\}$. 
Write $\Phi(d_1)\triangleq \alpha/q \|d_1-d\|_\star^q$. 
Combining the uniform smoothness assumption on $f^*$, adding and subtracting $\langle y; d\rangle$ and with the equality $f^*(d) + f(x) = \langle d ; x \rangle$ \citep[Theorem 23.5.]{rockafellar1970convex}, we have for any $y$
\begin{eqnarray*}
f(y) & \geq & \text{sup}_{d_1\in \mathbb{R}^m} \big\{ \langle y; d_1\rangle - \big(f^*(d) + \langle x; d_1 - d\rangle + \Phi(d_1 - d) \big)\big\}\\
f(y) & \geq & \text{sup}_{d_1\in \mathbb{R}^m} \big\{ \langle y; d_1 -d\rangle - \big(f^*(d) + \langle x; d_1 - d\rangle + \Phi(d_1 - d) \big) + \langle y; d \rangle \big\}\\
f(y) & \geq & \text{sup}_{d_1\in \mathbb{R}^m} \big\{ \langle y - x ; d_1 -d\rangle - \Phi(d_1 - d) + \langle y; d \rangle - f^*(d) \big\}\\
f(y) & \geq & \text{sup}_{d_1\in \mathbb{R}^m} \big\{ \langle y-x; d_1 -d\rangle - \Phi(d_1 -d) \big\} + \langle y; d \rangle + f(x) - \langle d; x\rangle\\
f(y) & \geq & f(x) + \langle d ; y - x \rangle + \Phi^*(y - x).
\end{eqnarray*}
Write $\|\cdot\|=\|\cdot\|_{\mathcal{C}}$ for some compact centrally symmetric convex set $\mathcal{C}$ with nonempty interior. 
Then, note that $\Phi = \frac{\alpha}{q}\sigma_{\mathcal{C}}^q$. 
Hence, by Lemma \ref{lem:power_norm_fenchel}, we have $\Phi^*(\cdot) = \|\cdot\|_{\mathcal{C}}^p/(p\alpha^{1/(q-1)})$. Finally
\begin{equation*}\tag*{\qedhere}
f(y) \geq f(x) + \langle d ; y - x \rangle + \frac{1}{p\alpha^{1/(q-1)}}\|y - x\|_{\mathcal{C}}^p.
\end{equation*}
\end{proof}

\subsection{Uniform Convexity and Smoothness for Sets and Spaces}\label{ssec:set_space_UC_UC}

Moduli of convexity and smoothness of a norm $\|\cdot\|_{\mathcal{C}}$ help characterize the geometry of the normed space $(\mathbb{R}^m, \|\cdot\|_{\mathcal{C}})$ or the convex set $\mathcal{C}$. This connects set uniform convexity with results in the study of Banach spaces, in the special case where $\mathcal{C}$ is centrally symmetric with nonempty interior. In Section \ref{ssec:kakade}, we provide an important use case stemming from this other perspective on uniform convexity. These moduli are classical objects characterizing either enhanced convex properties of $\mathcal{C}$ (for uniform convexity, rotundity) or regularity of the boundary of $\mathcal{C}$ (uniform smoothness). Here too, these properties are dual for a normed space and its dual space \citep{lindenstrauss1963modulus}. 

The \emph{(global) modulus of convexity} \citep{clarkson1936uniformly} is defined, for $\epsilon\in[0,2]$, as
\begin{equation}\label{eq:modulus_convexity}
\delta_{\|\cdot\|_\mathcal{C}}(\epsilon) = \text{inf}\{ 1 - \|(x+y)/2\|_{\mathcal{C}}~|~\|x\|_{\mathcal{C}} = \|y\|_{\mathcal{C}} = 1;~\|x-y\|_{\mathcal{C}}\geq \epsilon \}.
\end{equation}
The restriction of $\epsilon\in[0,2]$ ensures that the infimum is defined.
It measures the convexity of $\|\cdot\|_{\mathcal{C}}$ at \textit{midpoints} on the border of $\mathcal{C}$.
Note that the value of $\delta_{\|\cdot\|_\mathcal{C}}$ does not change by considering $(x,y)\in B_{\|\cdot\|}(1)$ in place of $S_{\|\cdot\|}(1)$, see discussion following \citep[Definition 1.e.1]{lindenstrauss2013classical}.
The modulus of smoothness \citep{lindenstrauss1963modulus} of $\|\cdot\|_\mathcal{C}$ is defined, for $\tau>0$, as
\begin{equation}\label{eq:modulus_smoothness}
\rho_{\|\cdot\|_\mathcal{C}}(\tau)  =  \text{sup}\{ (\|x + \tau y\|_{\mathcal{C}} + \|x - \tau y\|_{\mathcal{C}})/2- 1 ~|~ \|x\|_{\mathcal{C}} = \|y\|_{\mathcal{C}} =1\}.
\end{equation}
We can now define uniformly convex (resp. smooth) norm balls and normed spaces.

\begin{definition}[Uniformly Convex Set or Space]\label{def:various_definition_UC}
Consider a compact convex set $\mathcal{C}$, $p \geq 2$ and $\alpha>0$. 
Assume $\mathcal{C}$ is centrally symmetric with nonempty interior. 
$\mathcal{C}$ is $(\alpha, p)$-uniformly convex iff for any $\epsilon\in[0,2]$
\[
\delta_{\|\cdot\|_{\mathcal{C}}}(\epsilon) \geq \alpha \epsilon^p.
\]
In that case, we also say that the normed space $(\mathbb{R}^m, \|\cdot\|_{\mathcal{C}})$ is uniformly convex of type $p$.
\end{definition}

There are other equivalent definitions of the set uniform convexity of $\mathcal{C}$. We will detail some of them in Theorem \ref{th:UC_set_and_support}, prove their equivalence and discuss their practical significance.
 
\begin{definition}[Uniformly Smooth Set or Space]\label{def:set_US}
Consider a compact convex set $\mathcal{C}$ and $q \in ]1,2]$. 
Assume $\mathcal{C}$ is centrally symmetric with nonempty interior. 
$\mathcal{C}$ is $(\alpha, q)$-uniformly smooth if for any $\tau>0$, we have
\[
\rho_{\|\cdot\|_\mathcal{C}}(\tau) \leq \alpha \tau^q.
\]
In that case, we also say that the normed space $(\mathbb{R}^m,\|\cdot\|_{\mathcal{C}})$ is uniformly smooth of type $q$.
\end{definition}

When a set is $(\mu, 2)$-uniformly convex (resp. $(L, 2)$-uniformly smooth), we say it is $\mu$-strongly convex (resp. $L$-smooth), see \citep[Theorem 2.1.]{goncharov2017strong} for a thorough review on strongly convex sets in Hilbert spaces.
These properties are dual to each other, in terms of the set $\mathcal{C}$ and its polar $\mathcal{C}^\circ$, or the norm ball and its dual norm ball \citep[Proposition IV 1.12]{deville1993smoothness}.
The Lindenstrauss formula \citep[Theorem 1]{lindenstrauss1963modulus} leads to quantitative versions of that duality. 
For any $\tau>0$, we have
\BEQ\label{eq:global_lindstrauss_formula}\tag{Lindenstrauss}
\rho_{\|\cdot\|_{\mathcal{C}^\circ}}(\tau) = \text{sup}_{\epsilon\in[0,2]}\Big\{ \frac{\tau\epsilon}{2} - \delta_{\|\cdot\|_\mathcal{C}}(\epsilon)\Big\}.
\EEQ
The following lemma \citep[Proposition 1.12.]{deville1993smoothness} then quantifies this duality and is similar to Proposition \ref{prop:Duality_Fenchel_UC_US} on a function and its Fenchel conjugate.
The proof directly follows from \eqref{eq:global_lindstrauss_formula}.

\begin{proposition}[Uniform Smoothness and Convexity with dual norms]\label{prop:Duality_Polar_UC_US}
Consider $\alpha,c>0$, $p\geq 2$ and $q\in]1,2]$ such that $\frac{1}{p}+\frac{1}{q}=1$ and a compact convex set $\mathcal{C}$ centrally symmetric with nonempty interior.
We have the following implications
\begin{enumerate}[label=(\alph*)]
    \item If $\mathcal{C}$ is $(\alpha, q)$-uniformly smooth (Definition \ref{def:various_definition_UC}), then $\mathcal{C}^\circ$ is $(1/\big(2p(2\alpha q)^{1/(q-1)}\big), p)$-uniformly convex (Definition \ref{def:set_US}). \label{itm:US_Polar_UC}
    
    \item If $\mathcal{C}$ is $(c, p)$-uniformly convex, then $\mathcal{C}^\circ$ is $(1/\big(2q(2\alpha p)^{q-1}\big), q)$-uniformly smooth.
    \label{itm:UC_Polar_US}
\end{enumerate}
\end{proposition}
\begin{proof}[Proof of Proposition \ref{prop:Duality_Polar_UC_US}]
For instance, let us prove \ref{itm:US_Polar_UC}. With \eqref{eq:global_lindstrauss_formula}, we have for any $\tau>0$ and $\epsilon\in[0,2]$ that $\tau \epsilon/2 - \delta_{\|\cdot\|_{\mathcal{C}^\circ}}(\epsilon) \leq \alpha \tau^q$. Optimizing w.r.t. to $\tau$, the optimal $\tau^*=(\epsilon/(2\alpha q))^{1/(q-1)}$ leads to $\delta_{\|\cdot\|_{\mathcal{C}^\circ}}(\epsilon) \leq \frac{1}{2p}\frac{\epsilon^p}{(2\alpha q)^{1/(q-1)}}$.
\end{proof}

\subsection{Local Moduli.}
Local counterparts of the global moduli characterize local properties of $\mathcal{C}$ around a point $x^*\in\partial\mathcal{C}$ with respect to a (normalized) direction $d$ in the normal cone $N_{\mathcal{C}}(x^*)$.
As we will see, these local properties are important as they explain empirical globally accelerated convergence rates in optimization problems where the functions or constraints do not satisfy global regularity assumptions such as, \textit{e.g.}, strong convexity \citep{dunn1979rates,kerdreux2020uc}.\\
The \emph{local modulus of smoothness} \citep[(15)]{goncharov2017strong} of $\mathcal{C}$ at $x^*\in\partial\mathcal{C}$ with respect to $d\in S_{\|\cdot\|_{\mathcal{C}^\circ}}(1)$ is defined as, for $t>0$,
\begin{equation}\label{eq:local_directional_modulus_smoothness}\tag{Loc. Smoothness}
\rho_{\|\cdot\|_\mathcal{C}}(t, x^*, d) = \text{sup }\big\{ \|x^* + t x\|_{\mathcal{C}} - \|x^*\|_{\mathcal{C}} - t \langle d; x\rangle ~ | \|x\|_{\mathcal{C}}\leq 1 \big\}. 
\end{equation}
Similarly to all moduli seen so far, the local modulus of smoothness is designed so that when $t$ goes to zero, the first order terms cancel.
In the following, for convenience we write $\rho_{\mathcal{C}}$ for $\rho_{\|\cdot\|_\mathcal{C}}$.
We measure the \emph{local uniform convexity} at $x^*$ via the local modulus of \textit{rotundity}.
In the equivalent characterization of the set uniform convexity, the definition of modulus of rotundity is most related with the \textit{scaling inequalities} characterizations, see \ref{itm:scaling_inequality} in Theorems \ref{th:UC_set_and_support} and \ref{thm:local_UC_smoothness_gauge}.
For $x^*\in\partial\mathcal{C}$ and $d\in N_{\mathcal{C}}(x^*)\cap S_{\|\cdot\|_{\mathcal{C}^\circ}}(1)$, the \emph{local modulus of rotundity} at $x^*$ w.r.t. $d$ is defined for $\epsilon \in [0,2]$ as
\begin{equation}\label{eq:modulus_rotundity}\tag{Rotundity}
\mathcal{\nu}_{\mathcal{C}}(\epsilon, x^*, d) = \text{inf} \big\{ \langle d; x^* - x\rangle~|~ x\in\mathcal{C},~\|x^*-x\|_{\mathcal{C}} \geq \epsilon   \big\}.
\end{equation}
The following lemma makes the duality between smoothness and rotundity explicit by linking the two moduli, to produce a local counterpart to~\eqref{eq:global_lindstrauss_formula}. We cite a version giving a quantitative dual relationship between local modulus of smoothness and local modulus of rotundity \citep[Theorem 2.7.]{goncharov2017strong}.

\begin{lemma}[Local Lindstrauss formula]\label{lem:local_lindstrauss_formula}
Consider $x^*\in\partial\mathcal{C}$ and $d\in S_{\|\cdot\|_{\mathcal{C}^\circ}}(1)\cap N_{\mathcal{C}}(x^*)$. 
Then the local modulus of smoothness and rotundity satisfy for any $t>0$
\begin{equation}\label{eq:local_lindstrauss_formula}\tag{Loc. Lindenstrauss}
\rho_{\mathcal{C}^\circ}(t, d, x^*) = \text{sup}_{\epsilon \in [0,2] }\big\{ \epsilon t - \nu_{\mathcal{C}}(\epsilon, x^*, d)\big\}.
\end{equation}
\end{lemma}
\begin{proof}[Proof of Lemma \ref{lem:local_lindstrauss_formula}]
Let $t>0$, by definition of $\rho_{\mathcal{C}^\circ}$, for $\eta>0$ there exists $d_{\eta}\in\mathcal{C}^\circ$ such that $\rho_{\mathcal{C}^\circ}(t, d, x^*) \leq \|d+ t d_\eta\|_{\mathcal{C}^\circ} - \|d\|_{\mathcal{C}^\circ} - t \langle x^*; d_\eta\rangle + \eta$.
Also, by compactness of $\mathcal{C}$, there exists $x_{\eta}\in\partial\mathcal{C}$ s.t. $\|t d_\eta + d\|_{\mathcal{C}^\circ}=\sigma_{\mathcal{C}}(t d_\eta + d) = \langle t d_\eta + d; x_\eta\rangle$.
Since $d\in N_{\mathcal{C}}(x^*)$, we have $\|d\|_{\mathcal{C}^\circ}=\sigma_{\mathcal{C}}(d) = \langle d; x^*\rangle$ and hence
\begin{eqnarray*}
\rho_{\mathcal{C}^\circ}(t, d, x^*) & \leq &  \|t d_\eta +d\|_{\mathcal{C}^\circ} - \|d\|_{\mathcal{C}^\circ} - t \langle x^* ; d_{\eta} \rangle + \eta\\
\rho_{\mathcal{C}^\circ}(t, d, x^*) & \leq &  \langle t d_\eta + d; x_\eta \rangle - \langle d; x^* \rangle - t \langle x^* ; d_{\eta} \rangle + \eta\\
\rho_{\mathcal{C}^\circ}(t, d, x^*) & \leq & \langle d; x_\eta - x^*\rangle + \langle t d_\eta; x_\eta - x^*\rangle + \eta \leq \langle d; x_\eta - x^*\rangle + t\sigma_{\mathcal{C}^\circ}(x_\eta - x^*) + \eta\\
\rho_{\mathcal{C}^\circ}(t, d, x^*) & \leq & \text{sup}_{x\in\mathcal{C}} \big\{\langle d; x - x^*\rangle + t \|x - x^*\|_{\mathcal{C}}\big\} + \eta\\
\rho_{\mathcal{C}^\circ}(t, d, x^*) & \leq & \text{sup}_{\epsilon\in[0,2]}\text{sup}_{x\in\mathcal{C}} \big\{\langle d; x - x^*\rangle + t \|x - x^*\|_{\mathcal{C}}~\big|~ \|x - x^*\|_{\mathcal{C}}=\epsilon\big\} + \eta\\
\rho_{\mathcal{C}^\circ}(t, d, x^*) & \leq & \text{sup}_{\epsilon\in[0,2]} \big\{ t \epsilon - \text{inf}_{x\in\mathcal{C}} \big\{\langle d; x^*-x\rangle ~|~\|x - x^*\|_{\mathcal{C}}=\epsilon \big\}\big\} + \eta\\
\rho_{\mathcal{C}^\circ}(t, d, x^*) & \leq & \text{sup}_{\epsilon\in[0,2]} \big\{ t \epsilon - \nu_{\mathcal{C}}(\epsilon, x^*, d)\big\} + \eta.
\end{eqnarray*}
We used that for any $x,y\in\mathcal{C}$ we have  $\|x-y\|_{\mathcal{C}}\in[0,2]$.
Finally, last inequality is true for any $\eta>0$, hence $\rho_{\mathcal{C}^\circ}(t, d, x^*) \leq  \text{sup}_{\epsilon\in[0,2]} \big\{ t \epsilon - \nu_{\mathcal{C}}(\epsilon, x^*, d)\big\}$. 
We now provide a similar reasoning to obtain the equality.
Indeed, for $\lambda>0$, there exists $\epsilon_{\lambda}>0$ such that $\text{sup}_{\epsilon\in[0,2]} \big\{ t \epsilon - \nu_{\mathcal{C}}(\epsilon, x^*, d)\big\} \leq t \epsilon_{\lambda} - \nu_{\mathcal{C}}(\epsilon_{\lambda}, x^*, d) + \lambda$.
Also, for $\eta>0$, there exists $x_{\eta}\in\mathcal{C}$ s.t. 
$\nu_{\mathcal{C}}(\epsilon_{\lambda}, x^*, d) \geq \langle d ; x^* - x_{\eta} \rangle - \eta$ with $\|x_{\eta}- x^*\|_{\mathcal{C}}\geq\epsilon_{\lambda}$.
By compactness of $\mathcal{C}$, there exists $d_\eta\in\mathcal{C}^\circ$ such that $\|x_{\eta} - x^*\|_{\mathcal{C}} = \sigma_{\mathcal{C}^\circ}(x^* - x_{\eta}) = \langle x^* - x_{\eta}; d_\eta\rangle$. Therefore, for $t>0$, we have 
\begin{eqnarray*}
\text{sup}_{\epsilon\in[0,2]} \big\{ t \epsilon - \nu_{\mathcal{C}}(\epsilon, x^*, d)\big\} &\leq& t \epsilon_{\lambda}  - \langle d; x^* - x_{\eta}\rangle + \lambda + \eta \leq t \|x_{\eta} - x^*\|_{\mathcal{C}}  - \langle d; x^* - x_{\eta}\rangle + \lambda + \eta \\
\text{sup}_{\epsilon\in[0,2]} \big\{ t \epsilon - \nu_{\mathcal{C}}(\epsilon, x^*, d)\big\} &\leq& t \langle d_\eta; x^* - x_{\eta} \rangle - \langle d; x^* - x_{\eta}\rangle + \lambda + \eta\\
\text{sup}_{\epsilon\in[0,2]} \big\{ t \epsilon - \nu_{\mathcal{C}}(\epsilon, x^*, d)\big\} &\leq& \langle x_{\eta}; d - t d_\eta\rangle - \langle d ; x^*\rangle + t \langle d_\eta; x^*\rangle + \lambda + \eta\\
\text{sup}_{\epsilon\in[0,2]} \big\{ t \epsilon - \nu_{\mathcal{C}}(\epsilon, x^*, p)\big\} &\leq& \sigma_{\mathcal{C}}(d - t d_\eta) - \sigma_{\mathcal{C}}(d) - t\langle x^*; -d_\eta \rangle + \lambda + \eta\\
\text{sup}_{\epsilon\in[0,2]} \big\{ t \epsilon - \nu_{\mathcal{C}}(\epsilon, x^*, d)\big\} &\leq& \|d - t d_\eta\|_{\mathcal{C}^\circ} - \|d\|_{\mathcal{C}^\circ} - t \langle x^*; -d_\eta  \rangle + \lambda + \eta.
\end{eqnarray*}
Hence, since $-d_\eta\in\mathcal{C}^\circ$, for any $\lambda, \eta>0$, we have $\nu_{\mathcal{C}}(\epsilon, x^*, d)\big\} \leq\rho_{\mathcal{C}^\circ}(t, d, x^*) \lambda + \eta$.
We conclude that $\text{sup}_{\epsilon\in[0,2]} \big\{ t \epsilon - \nu_{\mathcal{C}}(\epsilon, x^*, d)\big\} \leq \rho_{\mathcal{C}^\circ}(t, d, x^*)$.
\end{proof}

\section{Equivalence between Global Set and Functional Assumptions}\label{sec:global_equivalence}

We expose some classical equivalence between functional and geometrical properties in Theorem \ref{th:UC_set_and_support} below.
This leads to new insights in learning theory in Section \ref{ssec:kakade} and in optimization in Sections \ref{ssec:lojasiewicz}-\ref{ssec:PAFW}.

Item \ref{itm:ML_def_UC} is similar to the definition appearing in most machine learning papers \citep{garber2015faster,huang2016following,huang2017following} and gives an intuitive understanding of set uniform convexity. 
The uniformly \text{mid-convex} property is equivalent to its continuous counterpart, see, \textit{e.g.}, \citep[Lemma 9]{Molinaro20}, but allows more concise proofs.

Item \ref{itm:scaling_inequality} is an essential inequality in analysing projection-free online or offline optimization methods. 
There are other related and useful inequalities that can be seamlessly derived from this one, see, \textit{e.g.}, Lemma \ref{lem:assumption_verified}.

Item \ref{itm:sphere_support_smoothness}-\ref{itm:global_gauge_HEB_ball} provides equivalent functional properties of the gauge and support function of $\mathcal{C}$.
Note that $\mathcal{C}$ is UC, but it is only a power of its gauge $\|\cdot\|_{\mathcal{C}}$ that is UC in the sense of functions. 
Also, the support function is only \textit{partially} H\"older smooth as Item \ref{itm:sphere_support_smoothness} holds on the sphere $S_{\|\cdot\|_{\mathcal{C}^\circ}}(1)$. 
Again, it is only a specific power of the support function that is uniformly smooth in the sense of functions without restriction on its domain.

Finally, item \ref{itm:modulus_UC} connects all other perspectives with the study of uniformly convex Banach spaces.
This connection is rich with hindsights, see, \textit{e.g.}, Section \ref{ssec:kakade}.

These results are classical and appear in many textbooks \citep{deville1993smoothness,lindenstrauss2013classical} often in non-quantitative, scattered, or too generic forms. We detail self-contained elementary proofs and provide quantitative versions in the finite-dimensional setting. 
Further, we only present the most practically significant equivalent characterizations here.
In Section \ref{sec:local_equivalence}, we will provide similar quantitative results with local uniform convexity and smoothness of $\mathcal{C}$. 

\begin{theorem}[Global Set Uniform Convexity]\label{th:UC_set_and_support}
Consider $p\geq 2$ and $q\in]1,2]$ s.t. $\frac{1}{p} + \frac{1}{q}=1$.
Let $\mathcal{C}$ be a centrally symmetric compact convex set with nonempty interior.
The following assertions are equivalent
\begin{enumerate}[label=(\alph*)]
    \item (Set mid-convex property) There exists $\alpha>0$ s.t. for all $(x,y)\in\mathcal{C}$ we have 
    \[
    \frac{x+y}{2}+ \alpha \|x-y\|_{\mathcal{C}}^p B_{\|\cdot\|_{\mathcal{C}}}(1)\subset\mathcal{C}.
    \]
    \label{itm:ML_def_UC}
    
    \item (Global scaling inequality) There exists $\alpha>0$ s.t. for any $(x,y)\in \mathcal{C}\times\partial\mathcal{C}$ and $d\in\mathbb{R}^m$ with $d\in N_{\mathcal{C}}(y)$ (or $y\in\text{argmax}_{v\in\mathcal{C}}\langle d; v\rangle$) we have
    \begin{equation}\label{eq:global_scaling}\tag{Global-Scaling}
    \langle d ; y -x \rangle \geq \alpha \|d\|_{\mathcal{C}^\circ} \| y-x\|_{\mathcal{C}}^p.
    \end{equation}
    \label{itm:scaling_inequality}

    \item (Set Modulus UC) There exists $\alpha>0$ s.t. $\mathcal{C}$ is $(\alpha, p)$-uniformly convex (Definition \ref{def:various_definition_UC}), \textit{i.e.}, for any $\epsilon>0$, we have the following lower bound on the modulus \eqref{eq:modulus_convexity},
    \[
    \delta_{\|\cdot\|_\mathcal{C}}(\epsilon)\geq \alpha \epsilon^p.
    \]
    \label{itm:modulus_UC}
    
    \item (Support H\"older-Smooth Sphere)
    The exists $c>0$ s.t. the support function $\sigma_{\mathcal{C}}(\cdot)$ is $(c, q-1)$-H\"older smooth with respect to $\|\cdot\|_{\mathcal{C}^\circ}$ on $S_{\|\cdot\|_{\mathcal{C}^\circ}}(1)$, \textit{i.e.}, it is differentiable on $S_{\|\cdot\|_{\mathcal{C}^\circ}}$ and for any $(d_1,d_2) \in S_{\|\cdot\|_{\mathcal{C}^\circ}}(1)$, we have
    \begin{equation*}
    \big\|\nabla \sigma_{\mathcal{C}}(d_1) - \nabla \sigma_{\mathcal{C}}(d_2)\big\|_{\mathcal{C}} \leq c \big\|d_1 - d_2\big\|_{\mathcal{C}^\circ}^{q-1}= c \big\|d_1 - d_2\big\|_{\mathcal{C}^\circ}^{1/(p-1)}.
    \end{equation*}
    \label{itm:sphere_support_smoothness}
    
    \item (Support US)
    $\sigma_{\mathcal{C}}^q(\cdot)$ is differentiable on $\mathbb{R}^m$ and there exists $c>0$ s.t. $\sigma_{\mathcal{C}}^q(\cdot)$ is $(c, q)$-uniformly smooth on $\mathbb{R}^m$ with respect to $\|\cdot\|_{\mathcal{C}^\circ}$ for some $c>0$.
    \label{itm:global_support_smoothness}

    \item (Gauge UC) There exists $\alpha>0$ s.t. $\|\cdot\|_{\mathcal{C}}^p$ is $(\alpha, p)$-uniformly convex with respect to $\|\cdot\|_{\mathcal{C}}$ (Definition \ref{def:uniformly_convex_functions}).
    \label{itm:global_gauge_HEB_ball}
\end{enumerate}
\end{theorem}
\begin{proof}[Proof of Theorem \ref{th:UC_set_and_support}]

\ref{itm:ML_def_UC} $\implies$ \ref{itm:modulus_UC}. 
Let $(x,y)\in S_{\|\cdot\|_{\mathcal{C}}}(1)$. For $z=\frac{x+y}{\|x+y\|_{\mathcal{C}}}$, we have $\frac{x+y}{2}+\alpha\|x-y\|_{\mathcal{C}}^pz\in\mathcal{C}$. Hence
\[
\Big\|\frac{x+y}{2}\Big\|_{\mathcal{C}}\Big( 1 + \alpha \|x - y\|_{\mathcal{C}}^p \frac{2}{\|x + y\|_{\mathcal{C}}}\Big) \leq 1.
\]
This shows that $1 - \|(x+y)/2\|_{\mathcal{C}}\geq \alpha \|x - y\|^p_{\mathcal{C}}$ and hence $\delta_{\|\cdot\|_{\mathcal{C}}}(\epsilon) \geq \alpha \epsilon^p$.

\bigskip

\ref{itm:modulus_UC} $\implies$ \ref{itm:ML_def_UC}. 
Recall that the modulus of convexity $\rho_{\|\cdot\|_{\mathcal{C}}}(\epsilon)$ in~\eqref{eq:modulus_convexity}, can be written as the infimum over $(x,y)\in B_{\|\cdot\|}(1)$ instead of $S_{\|\cdot\|}(1)$, see discussion following \citep[Definition 1.e.1]{lindenstrauss2013classical}. 
Let $(x,y)\in\mathcal{C}$. 
By definition of the modulus of convexity, we have $1 - \|(x+y)/2\|_{\mathcal{C}} \geq \alpha \|x - y\|^p_{\mathcal{C}}$.
Hence by the triangle inequality, for any $z\in B_{\|\cdot\|_{\mathcal{C}}}(1)$, we have $\|(x+y)/2+ \alpha \|x-y\|_{\mathcal{C}}^p z\|_{\mathcal{C}} \leq 1$, so that $(x+y)/2 + \alpha \|x-y\|_{\mathcal{C}}^p z\in\mathcal{C}$.

\bigskip

\ref{itm:ML_def_UC} $\implies$ \ref{itm:scaling_inequality}. Let $x\in\mathcal{C}$, $y\in\partial\mathcal{C}$ and $d\in\mathbb{R}^m$ s.t. $d\in N_{\mathcal{C}}(y)$. 
We have $y\in\text{argmax}_{v\in\mathcal{C}}\langle d; v\rangle$. 
Because $(x+y)/2 + \alpha\|x-y\|_{\mathcal{C}}^p z\in\mathcal{C}$, for any $z\in B_{\|\cdot\|_{\mathcal{C}}}(1)$, the optimality of $y$ implies
\[
\langle d ; (x+y)/2 + \alpha\|x-y\|_{\mathcal{C}}^p z\rangle \leq \langle d; y\rangle.
\]
Hence, for any $z\in B_{\|\cdot\|_{\mathcal{C}}}(1)$ we have $ 2 \alpha \|x-y\|_{\mathcal{C}}^p \langle d; z \rangle \leq \langle d; y-x\rangle$. 
By definition of the dual norm, we hence have $2\alpha \|x-y\|_{\mathcal{C}}^p \|d\|_{\mathcal{C}}^\star \leq \langle d; y-x\rangle$ and conclude with $\|d\|_{\mathcal{C}}^\star = \|d\|_{\mathcal{C}^{\circ}}$.

\bigskip

\ref{itm:scaling_inequality} $\implies$ \ref{itm:sphere_support_smoothness}. 
Let $(d_1, d_2) \in S_{\|\cdot\|_{\mathcal{C}^\circ}}(1)$ and
consider $v_{d_i}\in\text{argmax}_{v\in\mathcal{C}}\langle d_i; v\rangle$ for $i=1,2$.
We have that for any $x\in\mathcal{C}$
\begin{equation*}
   \left\{
    \begin{split}
    \langle d_1; v_{d_1} - x\rangle &\geq \alpha \|d_1\|_{\mathcal{C}^\circ}\cdot \|v_{d_1} - x\|_{\mathcal{C}}^p = \alpha \|v_{d_1} - x\|_{\mathcal{C}}^p \\
    \langle d_2; v_{d_2} - x\rangle &\geq \alpha \|d_2\|_{\mathcal{C}^\circ}\cdot \|v_{d_2} - x\|_{\mathcal{C}}^p = \alpha \|v_{d_2} - x\|_{\mathcal{C}}^p.
    \end{split}
   \right.
\end{equation*}
Then, by summing the two inequalities evaluated respectively at $x=v_{d_2}$ and $x=v_{d_1}$, we have
\[
\langle d_1 - d_2; v_{d_1} - v_{d_2} \rangle \geq 2\alpha \|v_{d_1} - v_{d_2}\|_{\mathcal{C}}^p.
\]
By Cauchy-Schwartz and since $v_{d_i}=\nabla \sigma_{\mathcal{C}}(d_i)$ for $i=1,2$ (Lemma \ref{lem:differentiability_support} applies because $\mathcal{C}$ is strictly convex and $d_i\neq 0$), we obtain
\[
\|d_1 - d_2\|_{\mathcal{C}^\circ} \cdot \|\nabla \sigma_{\mathcal{C}}(d_1) - \nabla \sigma_{\mathcal{C}}(d_2)\|_{\mathcal{C}} \geq 2\alpha \|\nabla \sigma_{\mathcal{C}}(d_1) - \nabla \sigma_{\mathcal{C}}(d_2)\|_{\mathcal{C}}^p,
\]
and conclude that
\[
\|\nabla \sigma_{\mathcal{C}}(d_1) - \nabla \sigma_{\mathcal{C}}(d_2)\|_{\mathcal{C}} \leq \frac{1}{(2\alpha)^{1/(p-1)}} \|d_1 - d_2\|_{\mathcal{C}^\circ}^{1/(p-1)}.
\]
Note finally that $1/(p-1)= q-1$.

\bigskip

\ref{itm:global_support_smoothness} $\implies$ \ref{itm:modulus_UC}. 
Note that \cite[(d) $\implies$ (a) of Theorem 2.2.]{borwein2009uniformly} is not constructive and that \citep[(ii) $\implies$ (i) in Lemma 5.1.]{deville1993smoothness} is incomplete as it only proves that the modulus of smoothness has the right lower-bound for $\tau\in[0,1/2[$.
\cite{lindenstrauss2013classical} do not consider these aspects and \cite[\S 26]{kothe1983topological} neither.
\citep[(iii) of Theorem 3.7.4.]{zalinescu2002convex} bears some similarity.
Recall the duality between support and gauge functions $\sigma_C(\cdot) = \|\cdot\|_{\mathcal{C}^\circ}$. 
We now show that $\mathcal{C}^\circ$ is uniformly smooth by providing an upper bound on its modulus of smoothness and conclude on~\ref{itm:modulus_UC} by duality. 
Recall that for $\tau>0$, the modulus of smoothness of $\mathcal{C}^\circ$ is defined as
\[
\rho_{C^\circ}(\tau) = \text{sup}\big\{ \big(\|d_1 + \tau d_2\|_{\mathcal{C}^\circ} + \|d_1 - \tau d_2\|_{\mathcal{C}^\circ}\big)/2 - 1 ~ \big| ~ \|d_1\|_{\mathcal{C}^\circ} = \|d_2\|_{\mathcal{C}^\circ}=1\big\}.
\]
Consider $(d_1,d_2)\in S_{\|\cdot\|_{\mathcal{C}^\circ}}$, since $\sigma_{\mathcal{C}}^q$ is $(c,q)$-uniformly smooth on $\mathbb{R}^m$ and by equivalence between \ref{itm:fct_zero_order_smoothness} and \ref{itm:fct_first_order_smoothness} in Definition \ref{def:uniformly_smooth_functions}, we have
\begin{equation*}
   \left\{
    \begin{split}
    \|d_1 + \tau d_2\|_{\mathcal{C}^\circ}^q &\leq 1 + \langle \nabla \|\cdot\|^q_{\mathcal{C}^\circ}(d_1); \tau d_2 \rangle + \frac{2c}{q} \|\tau d_2\|_{\mathcal{C}^\circ}^{q}\\
    \|d_1 - \tau d_2\|_{\mathcal{C}^\circ}^q &\leq 1 - \langle \nabla \|\cdot\|^q_{\mathcal{C}^\circ}(d_1); \tau d_2 \rangle + \frac{2c}{q} \|\tau d_2\|_{\mathcal{C}^\circ}^{q}.
    \end{split}
   \right.
\end{equation*}
When $q\in]1,2]$, $(1+x)^{1/q}$ is concave and below its tangent.
In particular, $(1 + x)^{1/q}\leq 1 + x/q$. 
Hence, combined with $\|d_2\|_{\mathcal{C}^\circ}=1$, we have
\begin{equation*}
   \left\{
    \begin{split}
    \|d_1 + \tau d_2\|_{\mathcal{C}^\circ} &\leq 1 + \frac{1}{q}\langle \nabla \|\cdot\|^q_{\mathcal{C}^\circ}(d_1); \tau d_2 \rangle + \frac{2c}{q^2} \tau^{q}\\
    \|d_1 - \tau d_2\|_{\mathcal{C}^\circ} &\leq 1 - \frac{1}{q}\langle \nabla \|\cdot\|^q_{\mathcal{C}^\circ}(d_1); \tau d_2 \rangle + \frac{2c}{q^2} \tau^{q}.
    \end{split}
   \right.
\end{equation*}
Then summing the two inequalities and dividing by $2$, we obtain
\[
\big(\|d_1 + \tau d_2\|_{\mathcal{C}^\circ} + \|d_1 - \tau d_2\|_{\mathcal{C}^\circ}\big)/2 - 1 \leq \frac{2c}{q^2} \tau^{q}.
\]
Hence, $\mathcal{C}^\circ$ is $(2c/q^2,q)$-uniformly smooth.
Then Proposition \ref{prop:Duality_Polar_UC_US} \ref{itm:US_Polar_UC}, implies that $\mathcal{C}$ is $(1/(2p(2\alpha q)^{p-1}), p)$-uniformly convex with $\alpha=2c/q^2$, \text{i.e.}, $\mathcal{C}$ is $(q^{p-1}/(2^{2p-1}pc^{p-1}))$-uniformly convex.

\bigskip

\ref{itm:global_gauge_HEB_ball} $\implies$ \ref{itm:global_support_smoothness}
From Lemma \ref{lem:power_norm_fenchel}, we have that $\big(\|\cdot\|_{\mathcal{C}}^p\big)^\star(\cdot) = \Big[\frac{1}{p^{1/(p-1)}} - \frac{1}{p^q}\Big]\sigma_{\mathcal{C}}^q(\cdot)$. 
Then Item \ref{itm:UC_Fenchel_US} of Proposition \ref{prop:Duality_Fenchel_UC_US} implies that $\Big[\frac{p-1}{p^{q-1}}\Big]\sigma_{\mathcal{C}}^q(\cdot)$ is $(c^\prime, q)$-uniformly smooth on $\mathcal{C}$ with respect to $\|\cdot\|_{\mathcal{C}}^\star= \|\cdot\|_{\mathcal{C}^\circ}$ and $c^\prime=1/(qc^{q-1})$.
Hence, $\sigma_{\mathcal{C}}^q$ is $(\big[\frac{p^{q-1}}{(p-1)qc^{q-1}}\big], q)$-uniformly smooth.
Note also that by equivalence between \ref{itm:fct_zero_order_smoothness} and \ref{itm:fct_first_order_smoothness} in Definition \ref{def:uniformly_smooth_functions}, we have that $\sigma_{\mathcal{C}}^q$ is differentiable.

\bigskip

\ref{itm:global_support_smoothness} $\implies$ \ref{itm:global_gauge_HEB_ball}. Conversely, let us assume that $\sigma_{\mathcal{C}}^q$ is $(\alpha, q)$-uniformly smooth.
From Lemma \ref{lem:power_norm_fenchel}, we have that $\big(\sigma_{\mathcal{C}}^q\big)^\star(\cdot) = \Big[\frac{1}{q^{1/(q-1)}} - \frac{1}{q^p}\Big]\|\cdot\|_{\mathcal{C}}^p(\cdot)$.
And, with Proposition \ref{prop:Duality_Fenchel_UC_US} \ref{itm:US_Fenchel_UC}, $\Big[\frac{q-1}{q^{p-1}}\Big]\|\cdot\|_{\mathcal{C}}^p(\cdot)$ is $(c^\prime, p)$-uniformly convex with respect to $\|\cdot\|_{\mathcal{C}}$ with $c^\prime = 1/(p\alpha^{p-1})$.
Finally, we conclude that $|\cdot\|_{\mathcal{C}}^p(\cdot)$ is $(\big[\frac{q^{p-1}}{(q-1)p\alpha^{p-1}}\big], p)$-uniformly convex.

\bigskip

\ref{itm:sphere_support_smoothness} $\implies$ \ref{itm:global_support_smoothness}. 
Conversely, let us show that $\sigma_{\mathcal{C}}^q(\cdot)$ is uniformly smooth.
The proof follows that of \citep[Theorem 2.1.]{borwein2009uniformly}.
Let us start by showing that $\sigma_{\mathcal{C}}^q(\cdot)$ is differentiable on $\mathbb{R}^m$. 
For $d_1\in\mathbb{R}^m\setminus\{0\}$, we have $\nabla \sigma_{\mathcal{C}}^q(d_1) = q \|d_1\|^{q-1} \nabla \sigma_{\mathcal{C}}(d_1)$.
Because $\mathcal{C}$ is strictly convex, there is a unique $x_1\in\partial\mathcal{C}$ s.t. $d_1\in N_{\mathcal{C}}(x_1)$.
From \citep[Corollary 1.7.3.]{schneider2014convex}, we have that $\nabla \sigma_{\mathcal{C}}(d_1)=x_1$.
Because $q>1$, when $d_1$ converges to $0$, we have that $\nabla \sigma_{\mathcal{C}}^q(d_1)$ also converges to zero. 
Hence, $\sigma_{\mathcal{C}}^q$ is differentiable at zero with $\nabla \sigma_{\mathcal{C}}^q(0)=0$.\\
Let $(d_1,d_2)\in\mathbb{R}^m$ and $x_i\in\partial\mathcal{C}$ s.t. $d_i\in N_{\mathcal{C}}(x_i)$, \textit{i.e.}, $\nabla \sigma_{\mathcal{C}}(d_i)=x_i$. 
Because $\sigma_{\mathcal{C}}$ is H\"older smooth on $S_{\|\cdot\|_{\mathcal{C}^\circ}}$, we have $\|\nabla \sigma_{\mathcal{C}}(d_1) - \nabla \sigma_{\mathcal{C}}(d_2)\|_{\mathcal{C}} \leq c\| d_1/\|d_1\|_{\mathcal{C}^\circ} - d_2/\|d_2\|_{\mathcal{C}^\circ}\|_{\mathcal{C}^\circ}^{1/(q-1)}$.
We then obtain
\begin{eqnarray*}
\|\nabla \sigma_{\mathcal{C}}^q(d_1) - \nabla \sigma_{\mathcal{C}}^q(d_2)\|_{\mathcal{C}} & = & \| q \sigma_{\mathcal{C}}^{q-1}(d_1) \nabla \sigma_{\mathcal{C}}(d_1) - q \sigma_{\mathcal{C}}^{q-1}(d_2) \nabla \sigma_{\mathcal{C}}(d_2)\|_{\mathcal{C}}\\
&\leq & q \sigma_{\mathcal{C}}^{q-1}(d_1) \big\| \nabla \sigma_{\mathcal{C}}(d_1) - \nabla \sigma_{\mathcal{C}}(d_2)\big\|_{\mathcal{C}} + q \big\| \nabla \sigma_{\mathcal{C}}(d_2) \big\|_{\mathcal{C}} \big|\sigma_{\mathcal{C}}^{q-1}(d_1) - \sigma_{\mathcal{C}}^{q-1}(d_2) \big|\\
&\leq & qc\|d_1\|_{\mathcal{C}^\circ}^{q-1} \big\|d_1/\|d_1\|_{\mathcal{C}^\circ} - d_2/\|d_2\|_{\mathcal{C}^\circ}\big\|^{q-1}_{\mathcal{C}^\circ} + q \big| \|d_1\|_{\mathcal{C}^\circ}^{q-1} - \|d_2\|_{\mathcal{C}^\circ}^{q-1} \big|\\
&\leq & qc \big\|d_1 - d_2\big(\|d_1\|_{\mathcal{C}^\circ}/\|d_2\|_{\mathcal{C}^\circ}\big)\big\|^{q-1}_{\mathcal{C}^\circ} + q \big| \|d_1\|_{\mathcal{C}^\circ}^{q-1} - \|d_2\|_{\mathcal{C}^\circ}^{q-1} \big|.
\end{eqnarray*}
We have for $\lambda_1,\lambda_2>0$ and $r\in]0,1]$ $|\lambda_1^r - \lambda_2^r|\leq |\lambda_1 - \lambda_2|^r$ \citep[Lemma 2.1.]{borwein2009uniformly}. Hence, for $q-1\in]0,1]$, we have $\big| \|d_1\|_{\mathcal{C}^\circ}^{q-1} - \|d_2\|_{\mathcal{C}^\circ}^{q-1} \big| \leq \big| \|d_1\|_{\mathcal{C}^\circ} - \|d_2\|_{\mathcal{C}^\circ}\big|^{q-1}\leq \|d_1 - d_2\|_{\mathcal{C}^\circ}^{q-1}$.
Also, with the triangle inequality $\big\|d_1 - d_2\big(\|d_1\|_{\mathcal{C}^\circ}/\|d_2\|_{\mathcal{C}^\circ}\big)\big\|\leq \|d_1 - d_2\|_{\mathcal{C}^\circ} + \|d_2\|_{\mathcal{C}^\circ} - \|d_1\|_{\mathcal{C}^\circ} \leq 2 \|d_1 - d_2\|_{\mathcal{C}^\circ}$.
Hence
\begin{equation*}
\|\nabla \sigma_{\mathcal{C}}^q(d_1) - \nabla \sigma_{\mathcal{C}}^q(d_2)\|_{\mathcal{C}} \leq q(c2^{q-1} + 1) \|d_1 - d_2\|_{\mathcal{C}^\circ}^{q-1}.
\end{equation*}
Equivalence between \ref{itm:fct_zero_order_smoothness} and \ref{itm:fct_holder_gradient_smoothness} in Definition \ref{def:uniformly_smooth_functions} shows that $\sigma_{\mathcal{C}}^q$ is $(2q^2(c2^{q-1}+1),q)$-uniformly smooth.

\bigskip

\ref{itm:global_support_smoothness} $\implies$ \ref{itm:sphere_support_smoothness}. Let $(d_1,d_2)\in S_{\|\cdot\|_{\mathcal{C}^\circ}}(1)$. 
Since, for $i=1,2$, $\nabla \sigma_{\mathcal{C}}^q(d_i)=q\sigma_{\mathcal{C}}(d_i)^{q-1} \nabla \sigma_{\mathcal{C}}(d_1)= q\sigma_{\mathcal{C}}(d_1)$, we directly have (because of the equivalence between \ref{itm:fct_zero_order_smoothness} and \ref{itm:fct_holder_gradient_smoothness} in Definition \ref{def:uniformly_smooth_functions})
\begin{equation*}\tag*{\qedhere}
\|\nabla \sigma_{\mathcal{C}}(d_1) - \nabla \sigma_{\mathcal{C}}(d_2)\|_{\mathcal{C}} \leq \frac{c}{q} \|d_1 - d_2\|_{\mathcal{C}^\circ}^{1/(p-1)}.
\end{equation*}
\end{proof}

\begin{remark}
From the proof of Theorem \ref{th:UC_set_and_support}, one can obtain quantitative results. 
\ref{itm:ML_def_UC} and \ref{itm:modulus_UC} are equivalent with the same constant. 
\ref{itm:ML_def_UC} with $(\alpha,p)$  implies \ref{itm:scaling_inequality} with $(2\alpha, p)$;
\ref{itm:scaling_inequality} with $(\alpha, p)$ implies \ref{itm:sphere_support_smoothness} with $(1/(2\alpha)^{q-1}, q-1)$;
\ref{itm:global_support_smoothness} with $(c,q)$ implies \ref{itm:modulus_UC} with $(q^{p-1}/(2^{2p-1}pc^{p-1}),p)$; 
\ref{itm:global_gauge_HEB_ball} with $(\alpha, p)$ implies \ref{itm:global_support_smoothness} with $(p^{q-1}/((p-1)qc^{q-1}), q)$;
Conversely, \ref{itm:global_support_smoothness} with $(\alpha, q)$ implies \ref{itm:global_gauge_HEB_ball} with $(q^{p-1}/((q-1)p\alpha^{p-1}), p)$;
Finally, \ref{itm:sphere_support_smoothness} with $(c, q-1)$ implies \ref{itm:global_support_smoothness} with $(2q^2(c2^{q-1}+1))$.
\end{remark}

\section{Equivalence between Local Set and Functional Assumptions}\label{sec:local_equivalence}
In this section, we provide equivalent characterizations of the \textit{local} uniform convexity of $\mathcal{C}$ at $x^*\in\partial\mathcal{C}$.
The results are summarized in Theorem \ref{thm:local_UC_smoothness_gauge}, the analog to Theorem \ref{th:UC_set_and_support}.
We seek to articulate different useful views on the local uniform convexity property of a set.\\
Item \ref{itm:local_UC} is a Banach geometry definition via the local modulus of rotundity.
Item \ref{itm:local_scaling_inequality} is a geometric local \textit{scaling inequality} useful in some algorithm analysis, see for instance the Frank-Wolfe method on locally uniformly convex sets \citep{kerdreux2020uc}.
Note that a natural local version of \eqref{eq:global_scaling}, could be that for any $d\in N_{\mathcal{C}}(x^*)$, for any $x\in\mathcal{C}$, we require
\[
\langle d; x^* - x \rangle \geq \alpha \|d\|_{\mathcal{C}^\circ} \|x^* - x\|^q_{\mathcal{C}}.
\]
However, we opted for a weaker version in~\eqref{eq:local_scaling} which expresses the property only with respect to a single direction in the normal cone at the point of interest.
Finally Items \ref{itm:local_scaling_inequality} and \ref{itm:local_gauge_HEB} connect these geometrical characterization with their functional counterpart, both in term of smoothness and uniform convexity.
These results appear scattered in the literature, see, \textit{e.g.}, \citep[Chapter 3.7]{zalinescu1983uniformly} or \citep[Proposition 3.2.]{aze1995uniformly}.
We expect these various equivalences to provide convergence proof of algorithms in online and offline settings when the decision sets or constraints sets are not globally strongly convex.
We provide an example of such a result in Section \ref{ssec:lojasiewicz}.

\begin{theorem}[Local Set Uniform Convexity]\label{thm:local_UC_smoothness_gauge}
Consider $p\geq 2$ and $q\in]1,2]$ s.t. $\frac{1}{p} + \frac{1}{q}=1$. 
Let $\mathcal{C}$ be a compact strictly convex set centrally symmetric with nonempty interior.
Let $x^*\in\partial\mathcal{C}$, $d_1\in N_{\mathcal{C}}(x^*)\cap S_{\|\cdot\|_{\mathcal{C}^\circ}}(1)$ (note $S_{\|\cdot\|_{\mathcal{C}^\circ}}(1)=\partial\mathcal{C}^\circ$).
The following assertions are equivalent
\begin{enumerate}[label=(\alph*)]
    \item (Modulus of Rotundity) There exists $\alpha>0$ s.t. $\mathcal{C}$ is $(\alpha, p)$-locally uniformly convex at $x^*$ w.r.t. direction $d_1$, \textit{i.e.}, for any $\epsilon\in[0,2]$, we have
    \[
    \nu_{\mathcal{C}}(\epsilon, x^*, d_1) \triangleq \text{inf } \big\{ \langle d_1; x^* - x\rangle~|~ x\in\mathcal{C},~\|x-x^*\|_{\mathcal{C}} \geq \epsilon \big\} \geq \alpha \epsilon^p.
    \]
    \label{itm:local_UC}

    \item (Local scaling inequality) For any $x\in\mathcal{C}$, we have
    \begin{equation}\label{eq:local_scaling}\tag{Local-Scaling}
    \langle d_1 ; x^* - x \rangle \geq \alpha \|x^* - x\|_{\mathcal{C}}^p.
    \end{equation}
    \label{itm:local_scaling_inequality}

    \item (Support Local H\"older-Smooth Sphere) 
    There exists $c>0$ s.t. $\sigma_{\mathcal{C}}(\cdot)$ is $(c, q-1)$-H\"older smooth at $d_1$ on $S_{\|\cdot\|_{\mathcal{C}^\circ}}(1)$ w.r.t. $\|\cdot\|_{\mathcal{C}^\circ}$, \textit{i.e.}, for any $d_2\in S_{\|\cdot\|_{\mathcal{C}^\circ}}(1)$, we have
    \[
        \big\|\nabla \sigma_{\mathcal{C}}(d_1) - \nabla \sigma_{\mathcal{C}}(d_2)\big\|_{\mathcal{C}} \leq c \big\|d_1 - d_2\big\|_{\mathcal{C}^\circ}^{q-1}=\big\|d_1 - d_2\big\|_{\mathcal{C}^\circ}^{1/(p-1)}.
    \]
    \label{itm:local_Holder_smooth}

    \item (Support Local US) There exists $c>0$ s.t. $\sigma_{\mathcal{C}}^q(\cdot)$ is $(\alpha, q)$-uniformly smooth at $d_1$ w.r.t $\|\cdot\|_{\mathcal{C}^\circ}$, \textit{i.e.}, for any $d_2\in\mathbb{R}^m$, we have
    \[
    \sigma_{\mathcal{C}}^q(d_2) \leq \sigma_{\mathcal{C}}^q(d_1) + q \langle x^*; d_2 - d_1\rangle + \frac{\alpha}{q} \|d_2 - d_1\|^q_{\mathcal{C}^\circ},
    \]
    where $\nabla \sigma_{\mathcal{C}}^q(d_1) = q x^*$.
    \label{itm:local_support_US}
    
    \item  (Gauge local UC) There exists $\mu>0$ s.t. $\|\cdot\|^p_{\mathcal{C}}$ is $(\mu, p)$-uniformly convex at $x^*$ on $\mathcal{C}$ in direction $d_1$ w.r.t. $\|\cdot\|_{\mathcal{C}}$, \textit{i.e.}, for any $y\in\mathbb{R}^m$
    \[
    \|y\|^p_{\mathcal{C}} \geq \|x^* \|^p_{\mathcal{C}} + p\langle d_1; y-x^*\rangle + \frac{\mu}{2} \|y-x^*\|_{\mathcal{C}}^p.
    \]
    \label{itm:local_gauge_HEB}
\end{enumerate}
\end{theorem}

\begin{proof}[Proof of Theorem \ref{thm:local_UC_smoothness_gauge}]
Because $\mathcal{C}$ is strictly convex, $\sigma_{\mathcal{C}}$ is differentiable on $\mathbb{R}^m\setminus\{0\}$, see Lemma \ref{lem:differentiability_support}. In particular, $\nabla\sigma_{\mathcal{C}}(d_1) = x^*$ since $d_1\in N_{\mathcal{C}}(x^*)$. 
Also, because $\|d_1\|_{\mathcal{C}^\circ}=1$, note that $\nabla\sigma_{\mathcal{C}}^q(d_1)=q\|d_1\|_{\mathcal{C}^\circ}^{q-1} \nabla \sigma_{\mathcal{C}}(d_1)= q x^*$
Finally, note that $\|\cdot\|_{\mathcal{C}}=\sigma_{\mathcal{C}^\circ}$ is not necessarily differentiable (would require assuming that $\mathcal{C}^\circ$ is smooth).


\ref{itm:local_UC} $\iff$ \ref{itm:local_scaling_inequality} is immediate.

\bigskip

\ref{itm:local_UC} $\implies$ \ref{itm:local_Holder_smooth}.
Let us assume that $\mathcal{C}$ is $(\alpha, p)$-uniformly convex at $x^*\in\partial\mathcal{C}$ with respect to $d_1\in S_{\|\cdot\|_{\mathcal{C}^\circ}}(1)\cap N_{\mathcal{C}}(x^*)$, \textit{i.e.}, for any $\epsilon>0$, $\nu_{\mathcal{C}}(\epsilon, x^*, d_1)\geq \alpha \epsilon^p$. 
Hence, we have for any $x\in\mathcal{C}$
\[
\langle d_1 ; x^* - x \rangle \geq \alpha \|x - x^*\|_{\mathcal{C}}^p.
\]
Let $d_2\in S_{\|\cdot\|_{\mathcal{C}^\circ}}(1)$ and $x_2\triangleq \text{argmax}_{x\in\mathcal{C}}\langle x ; d_2 \rangle$ (it is unique because $\mathcal{C}$ is strictly convex compact). 
In particular, $\langle x^*-x_2 ; d_2 \rangle\leq 0$, hence we have
\[
\langle d_1 - d_2; x^* - x_2\rangle \geq \langle d_1 - d_2; x^* - x_2\rangle +  \underbrace{\langle d_2; x^* - x_2\rangle}_{ \leq 0} = \langle d_1; x^* - x_2\rangle \geq \alpha \| x^* - x_2\|_{\mathcal{C}}^p.
\]
Then, with Cauchy-Schwartz we have $\|d_1-d_2\|_{\mathcal{C}^\circ} \|x^* - x_2\|_{\mathcal{C}} \geq \alpha \| x^* - x_2\|_{\mathcal{C}}^p$.
Hence,
\[
\|x_2 - x^*\|_{\mathcal{C}}\leq \frac{1}{\alpha^{1/(p-1)}} \|d_1 - d_2\|_{\mathcal{C}^\circ}^{1/(p-1)}.
\]
With Lemma \ref{lem:differentiability_support}, we have $\nabla \sigma_{\mathcal{C}}(d_2)=x_2$ and $x^*=\nabla\sigma_{\mathcal{C}}(d_1)$, which concludes with $ q-1 = 1/(p-1)$.

\bigskip

\noindent \ref{itm:local_support_US} $\implies$ \ref{itm:local_UC}. 
Let us now assume that $\sigma_{\mathcal{C}}^q(\cdot)$ is $(\alpha, q)$-uniformly smooth at $d_1$ w.r.t $\|\cdot\|_{\mathcal{C}^\circ}$.
Also $\sigma_{\mathcal{C}}(\cdot)=\|\cdot\|_{\mathcal{C}^\circ}$.
Let us first prove an upper bound on the local modulus of smoothness $\rho_{\mathcal{C}^\circ}(t, d_1, x^*)$ of $\mathcal{C}^\circ$ at $d_1$ w.r.t. $x^*$, see \eqref{eq:local_directional_modulus_smoothness}.
By the duality formula \eqref{eq:local_lindstrauss_formula}, we will then obtain a lower bound on the modulus of rotundity.
Recall that the local modulus of smoothness in \eqref{eq:local_directional_modulus_smoothness} is defined for any $t>0$, as
\[
\rho_{\mathcal{C}^\circ}(t, d_1, x^*) = \text{sup }\big\{ \|d_1 + t d_2\|_{\mathcal{C}^\circ} - \|d_1\|_{\mathcal{C}^\circ} - t \langle x^*; d_2\rangle ~ | ~d_2\in\mathcal{C}^\circ \big\}.
\]
By \ref{itm:local_support_US}, 
we have for any $d_2\in\mathbb{R}^m$
\[
\|d_1 + t d_2\|_{\mathcal{C}^\circ}^q \leq \|d_1\|_{\mathcal{C}^\circ}^q + t\langle \nabla \sigma_{\mathcal{C}}^q(d_1); d_2\rangle + \frac{\alpha}{q}t^q\|d_2\|^q_{\mathcal{C}^\circ}.
\]
Recall from the beginning of the proofs that $\nabla \sigma_{\mathcal{C}}^q(d_1)=qx^*$. Then, we have by concavity of $(1+x)^{1/q}$ when $q\in ]1,2]$
\[
\|d_1 + t d_2\|_{\mathcal{C}^\circ} \leq \Big(1 + tq\langle x^*; d_2\rangle + \frac{\alpha}{q}t^q\Big)^q \leq 1 + t\langle x^*; d_2\rangle + \frac{\alpha}{q^2}t^q.
\]
In particular, for $d_2\in\mathcal{C}^\circ$ and because $\|d_1\|_{\mathcal{C}^\circ}=1$, we have $\rho_{\mathcal{C}^\circ}(t, d_1, x^*) \leq \alpha/q^2 t^q$.
Then, with Lemma \ref{lem:local_lindstrauss_formula}, we have that for any $\epsilon\in[0,2]$ and $t>0$
\[
\text{sup}_{\epsilon \in[0,2]}\big\{ \epsilon t - \nu_{\mathcal{C}}(\epsilon, x^*, d_1)\big\}\leq \alpha/q^2 t^{q}.
\]
Hence, for any $\epsilon\in[0,2]$
\[
\nu_{\mathcal{C}}(\epsilon, x^*, d_1) \geq \epsilon t - \alpha/q^2 t^{q}.
\]
Then for $t= (q\epsilon/\alpha)^{1/(q-1)}$, we have
\[
\nu_{\mathcal{C}}(\epsilon, x^*, d_1) \geq \frac{q^{p-2}}{\alpha^{p-1}}(q-1)\epsilon^p.
\]
Therefore, $\mathcal{C}$ is $(\frac{q^{p-2}}{\alpha^{p-1}}(q-1), p)$-locally uniformly convex at $x^*$ with respect to $d_1$.

\bigskip

\noindent \ref{itm:local_Holder_smooth} $\implies$ \ref{itm:local_support_US}. 
The proof is similar to that of \ref{itm:sphere_support_smoothness} $\implies$ \ref{itm:global_gauge_HEB_ball} in Theorem \ref{th:UC_set_and_support}, we repeat it for completeness.
First, by the very same argument, $\sigma_{\mathcal{C}}^q$ is differentiable on $\mathbb{R}^m$ (recall that $\sigma_{\mathcal{C}}$ is not differentiable at $0$).
Now, consider $d_2\in \mathbb{R}^m\setminus\{0\}$ and the unique (because $\mathcal{C}$ is strictly convex) $x_2\in\partial\mathcal{C}$ s.t. $d_2\in N_{\mathcal{C}}(x_2)$. 
Then, with Lemma \ref{lem:differentiability_support}, we have $\nabla \sigma_{\mathcal{C}}(d_2)=x_2$ and with the same argument $\nabla \sigma_{\mathcal{C}}(d_2/\|d_2\|_{\mathcal{C}^\circ})=x_2$. 
Because $\sigma_{\mathcal{C}}$ is H\"older smooth at $d_1$ on $S_{\|\cdot\|_{\mathcal{C}^\circ}}$, we have $\|\nabla \sigma_{\mathcal{C}}(d_1) - \nabla \sigma_{\mathcal{C}}(d_2)\|_{\mathcal{C}} \leq c\| d_1 - d_2/\|d_2\|_{\mathcal{C}^\circ}\|_{\mathcal{C}^\circ}^{1/(q-1)}$.
We then obtain, by adding and subtracting $q\sigma_{\mathcal{C}}^{q-1}(d_1)\nabla\sigma_{\mathcal{C}}(d_2)$ and applying the triangle inequality
\begin{eqnarray*}
\|\nabla \sigma_{\mathcal{C}}^q(d_1) - \nabla \sigma_{\mathcal{C}}^q(d_2)\|_{\mathcal{C}} & = & \| q \sigma_{\mathcal{C}}^{q-1}(d_1) \nabla \sigma_{\mathcal{C}}(d_1) - q \sigma_{\mathcal{C}}^{q-1}(d_2) \nabla \sigma_{\mathcal{C}}(d_2)\|_{\mathcal{C}}\\
&\leq & q \sigma_{\mathcal{C}}^{q-1}(d_1) \big\| \nabla \sigma_{\mathcal{C}}(d_1) - \nabla \sigma_{\mathcal{C}}(d_2)\big\|_{\mathcal{C}} + q \big\| \nabla \sigma_{\mathcal{C}}(d_2) \big\|_{\mathcal{C}} \big|\sigma_{\mathcal{C}}^{q-1}(d_1) - \sigma_{\mathcal{C}}^{q-1}(d_2) \big|\\
&\leq & qc\|d_1\|_{\mathcal{C}^\circ}^{q-1} \big\|d_1/\|d_1\|_{\mathcal{C}^\circ} - d_2/\|d_2\|_{\mathcal{C}^\circ}\big\|^{q-1}_{\mathcal{C}^\circ} + q \big| \|d_1\|_{\mathcal{C}^\circ}^{q-1} - \|d_2\|_{\mathcal{C}^\circ}^{q-1} \big|\\
&\leq & qc \big\|d_1 - d_2\big(\|d_1\|_{\mathcal{C}^\circ}/\|d_2\|_{\mathcal{C}^\circ}\big)\big\|^{q-1}_{\mathcal{C}^\circ} + q \big| \|d_1\|_{\mathcal{C}^\circ}^{q-1} - \|d_2\|_{\mathcal{C}^\circ}^{q-1} \big|.
\end{eqnarray*}
We have for $\lambda_1,\lambda_2>0$ and $r\in]0,1]$ $|\lambda_1^r - \lambda_2^r|\leq |\lambda_1 - \lambda_2|^r$.
Hence, for $q-1\in]0,1]$, we have $\big| \|d_1\|_{\mathcal{C}^\circ}^{q-1} - \|d_2\|_{\mathcal{C}^\circ}^{q-1} \big| \leq \big| \|d_1\|_{\mathcal{C}^\circ} - \|d_2\|_{\mathcal{C}^\circ}\big|^{q-1}\leq \|d_1 - d_2\|_{\mathcal{C}^\circ}^{q-1}$.
Also, by the triangle inequality, $\big\|d_1 - d_2\big(\|d_1\|_{\mathcal{C}^\circ}/\|d_2\|_{\mathcal{C}^\circ}\big)\big\|\leq \|d_1 - d_2\|_{\mathcal{C}^\circ} + \|d_2\|_{\mathcal{C}^\circ} - \|d_1\|_{\mathcal{C}^\circ} \leq 2 \|d_1 - d_2\|_{\mathcal{C}^\circ}$.
Hence for any $d_2\in\mathbb{R}^m\setminus\{0\}$
\begin{equation*}
\|\nabla \sigma_{\mathcal{C}}^q(d_1) - \nabla \sigma_{\mathcal{C}}^q(d_2)\|_{\mathcal{C}} \leq q(c2^{q-1} + 1) \|d_1 - d_2\|_{\mathcal{C}^\circ}^{q-1}.
\end{equation*}
Let us now prove that this implies a first-order type definition of local smoothness.
For any $d_2$, by the mean value theorem, there exists $\lambda\in]0,1[$ such that
\begin{eqnarray*}
\sigma_{\mathcal{C}}^q(d_2) - \sigma_{\mathcal{C}}^q(d_1) &=& \langle \nabla \sigma_{\mathcal{C}}^q(\lambda d_1 + (1-\lambda) d_2); d_2 - d_1\rangle\\
&=& \langle \nabla\sigma^q_{\mathcal{C}}(d_1); d_2 - d_1\rangle + \langle \nabla \sigma_{\mathcal{C}}^q(\lambda d_1 + (1-\lambda) d_2) - \nabla\sigma^q_{\mathcal{C}}(d_1); d_1 - d_2\rangle\\
&\leq& \langle \nabla\sigma^q_{\mathcal{C}}(d_1); d_2 - d_1\rangle + \| \nabla \sigma_{\mathcal{C}}^q(\lambda d_1 + (1-\lambda) d_2) - \nabla\sigma^q_{\mathcal{C}}(d_1)\|_{\mathcal{C}} \| d_1 - d_2\|_{\mathcal{C}^\circ}\\
&\leq& \langle \nabla\sigma^q_{\mathcal{C}}(d_1); d_2 - d_1\rangle + q(c2^{q-1}+1) \| d_1 - d_2\|_{\mathcal{C}^\circ}^{q}.
\end{eqnarray*}
Hence $\sigma_{\mathcal{C}}^q$ is $(q(c2^{q-1}+1),q)$-uniformly convex at $d_1$ w.r.t. $\|\cdot\|_{\mathcal{C}^\circ}$.

\bigskip

Equivalence between \ref{itm:local_support_US} and \ref{itm:local_gauge_HEB} stems from Proposition \ref{prop:local_directional_duality}. 
Indeed, from Lemma \ref{lem:power_norm_fenchel} we have that $(\frac{1}{q}\sigma_{\mathcal{C}}^q)^\star(\cdot) = \frac{1}{p}\|\cdot\|_{\mathcal{C}}^p$.
Then, because $(x^*, d_1)\in \partial\mathcal{C}\times N_{\mathcal{C}}(x^*)\cap\partial\mathcal{C}^\circ$, we have $(x^*, d_1)\in\partial \frac{1}{q}\sigma_{\mathcal{C}}^q(d_1)\times\partial \frac{1}{p}\|\cdot\|_{\mathcal{C}}^p(x^*)$ and we can indeed apply Proposition \ref{prop:local_directional_duality}.
\end{proof}

\section{Applications}\label{sec:applications}

Theorems \ref{th:UC_set_and_support} and \ref{thm:local_UC_smoothness_gauge} offer different points of view on uniform convexity properties which yield improved rates in optimization or learning. We now detail three situations where the equivalence relationships detailed above lead to new results.

In Section \ref{ssec:lojasiewicz}, we show that the $\ell_p$ balls with $p>2$ are locally strongly convex on some points of their boundaries, while not being globally strongly convex. This leads to novel linear convergence  results for vanilla Frank-Wolfe algorithm on some curved sets that are not strongly convex.

In Section \ref{ssec:kakade}, we leverage a result on the geometry of Banach spaces, showing the inclusion of uniformly convex spaces into Rademacher spaces of type $q$. The equivalence between the UC of norms balls and space UC then implies generalization bounds on low norm linear predictors.

In Section \ref{ssec:PAFW}, we show how the Primal Averaging Frank-Wolfe algorithm \citep[Algorithm 4]{lan2013complexity} exhibits accelerated sublinear rates w.r.t. the $\mathcal{O}(1/T)$ baseline when the constraint set is uniformly convex and $\text{inf}_{x\in\mathcal{C}}\|\nabla f(x)\|> c >0$. 
The sublinear rates are slower than those of Frank-Wolfe with exact line-search or short-steps on uniformly convex sets but are obtained with (cheaper) pre-determined function agnostic step-sizes, and in fact oblivious of any structure of the problem.
To our knowledge, this is the only version of Frank-Wolfe achieving accelerated convergence w.r.t. $\mathcal{O}(1/T)$ with such agnostic step-sizes.

\subsection{Linear Convergence Rates for Vanilla Frank-Wolfe on Non-Strongly Convex Sets}\label{ssec:lojasiewicz}

Here, we apply Theorem \ref{thm:local_UC_smoothness_gauge} to derive accelerated convergence rates of algorithms solving the following constrained optimization problem
\begin{equation}\label{eq:opt}\tag{OPT}
\underset{x\in\mathcal{C}}{\text{minimize }} f(x),
\end{equation}
where $f$ is smooth convex function and $\mathcal{C}$ a compact convex set.
Write $x^*$ a solution of \eqref{eq:opt}.
\cite{kerdreux2020uc} shows that when a \textit{local scaling inequality} holds at $x^*$ with $p\geq 2$, $\alpha>0$, \textit{i.e.}, for any $x\in\mathcal{C}$
\begin{equation}\label{eq:local_scaling_optimization}
\langle - \nabla f(x^*); x^* -x \rangle \geq \alpha \|\nabla f(x^*)\|_\star\|x^* - x\|^p,
\end{equation}
then the vanilla Frank-Wolfe algorithm has an accelerated convergence rate compared to $\mathcal{O}(1/T)$. By optimality, $-\nabla f(x^*)\in N_{\mathcal{C}}(x^*)$, and \eqref{eq:local_scaling_optimization} is ensured when \ref{itm:local_scaling_inequality} in Theorem \ref{thm:local_UC_smoothness_gauge} holds.
While the local scaling inequalities are key to the convergence analyses, they are harder to check than the other equivalent conditions in Theorem \ref{thm:local_UC_smoothness_gauge}. In the following lemma, we show that although $\ell_p$ balls are not strongly convex when $p>2$, there are locally strongly convex (\textit{i.e.} $(\alpha,2)$-locally uniformly convex) at any $x^*\in\partial\ell_p(1)$ s.t. $\langle x^*;e_i\rangle\neq 0$ for all $i$, which means improved convergence rates in this subset of points. 

\begin{lemma}[Local Strong Convexity of the $\ell_p$ with $p > 2$]\label{lem:local_strong_convexity_lq}
Consider $p>2$ and $x=\sum_{i=1}^{m}\lambda_i e_i \in\partial\ell_p(1)$ s.t. $\lambda_i\neq 0$ for all $i\in[m]$. 
Then, there exists $\alpha>0$ s.t. $\ell_p(1)$ is $(\alpha, 2)$-locally uniformly convex at $x$.
\end{lemma}
\begin{proof}[Proof of Lemma \ref{lem:local_strong_convexity_lq}]
Let us write $\|\cdot\|_p$ the $\ell_p$ norm. 
With Theorem \ref{thm:local_UC_smoothness_gauge} \ref{itm:local_gauge_HEB}, we need to prove that $f(\cdot)\triangleq\|\cdot\|_p^2$ is $(\alpha, 2)$-uniformly convex at $x=\sum_{i=1}^{m}\lambda_i e_i \in\partial\ell_p(1)$ s.t. $\lambda_i\neq 0$ for all $i\in[m]$. Note that Item \ref{itm:local_gauge_HEB} of Theorem \ref{thm:local_UC_smoothness_gauge} requires a quadratic lower bound on $\mathbb{R}^m$. Here, we only prove it on a compact domain. 
However, equivalence with Item \ref{itm:local_scaling_inequality} of Theorem \ref{thm:local_UC_smoothness_gauge} is also valid with such a restriction. We omit the proof.
Without loss of generality, by central symmetry of $\ell_p$, let us assume that all $\lambda_i>0$. Note then that $\sum_{i}{\lambda_i^p}=1$.
$f$ is convex and twice differentiable at $x$. Let us first prove that the Hessian $H_f(x)$ has no zero eigenvalues. We have
\begin{equation*}
\left\{
    \begin{split}
    &\frac{\partial^2 f}{\partial x_{i_0}^2}(x) =  2(p-1) \lambda_{i_0}^{p-2} + 2(2-p) \lambda_{i_0}^{2p-2}\\
    &\frac{\partial^2 f}{\partial x_{i_0}\partial x_{j_0}}(x) = 2(2-p) (\lambda_{i_0}\lambda_{j_0})^{p-1}.
    \end{split}
   \right.
\end{equation*}
Hence, the Hessian of $f$ at $x$ is of the form
\[
H_f(x) = 2(p-1)\text{diag}(\lambda_1^{p-2}, \ldots, \lambda_m^{p-2}) + \Big( 2(2-p) (\lambda_{i}\lambda_{j})^{p-1} \Big)_{1\leq i,j \leq m}.
\]
Write $\Lambda =(\lambda_i)_{i=1,\ldots,m}$, we have that
\[
H_f(x) = 2(2-p)\Big[\frac{p-1}{2-p}\text{diag}(\Lambda^{p-2}) + (\Lambda^{p-1})^T \Lambda^{p-1}\Big].
\]
Then, note that for an invertible diagonal matrix $D=\text{diag}(d_1,\ldots,d_m)$ and vector $h=(h_1,\ldots,h_m)$, we have 
\[
\text{det}\big(D + h^T h\big) = \text{det}\big(D\big) \text{det}\big(I_m + D^{-1}h^T h\big) = \Big( 1 + \sum_{i=1}^{m}{\frac{h_i^2}{d_i}}\Big) \prod_{i=1}^{m}{d_i}.
\]
We then have
\begin{eqnarray*}
\text{det}\big(H_f(x)\big) & = & (2(2-p))^m\Big(\frac{p-1}{2-p}\Big)^m \Big( 1 + \frac{2-p}{p-1} \sum_{i=1}^{m}{\lambda_i^{2(p-1)}/\lambda_i^{p-2}}\Big) \prod_{i=1}^{m}{\lambda_i^{p-2}}\\
\text{det}\big(H_f(x)\big) & = & \big[2(p-1)\big]^m \Big[1 + \frac{2-p}{p-1} \sum_{i=1}^{m}{\lambda_i^p}\Big] \prod_{i=1}^{m}{\lambda_i^{p-2}} = 2^{m}\big(p-1\big)^{m-1} \prod_{i=1}^{m}{\lambda_i^{p-2}}>0,
\end{eqnarray*}
so that $H_f(x)\succ 0$. This ensures that on the compact domain $\mathcal{C}$, there exists a value $\mu>0$ s.t. for any $y\in\mathcal{C}$
\[
\|y\|^2_{\mathcal{C}} \geq \|x \|^p_{\mathcal{C}} + \langle \nabla \|\cdot\|^2_{\mathcal{C}}(x); y-x\rangle + \frac{\mu}{2} \|y-x^*\|_{\mathcal{C}}^p.
\]
This corresponds to  
Theorem \ref{thm:local_UC_smoothness_gauge} \ref{itm:local_gauge_HEB}.
\end{proof}

When $p>2$, the $\ell_p$ balls are not globally strongly convex. 
However, Lemma \ref{lem:local_strong_convexity_lq} shows that they are locally strongly convex on any boundary point which has no zero coordinates in the canonical basis.
In the following corollary, we show that this proves linear convergence rates of the vanilla Frank-Wolfe algorithm on $\ell_p$ balls (with $p\geq 2$) as analysed in \citep{kerdreux2020uc}.

\begin{corollary}[Linear Rates for FW on $\ell_p$ for $p>2$]\label{cor:accelerated_cv_FW}
Consider a convex smooth function $f$ such that $\text{inf}_{x\in\mathcal{C}}\|\nabla f(x)\|>c>0$ and $\mathcal{C}=\ell_p(1)$.
Assume the solution $x^*$ of \eqref{eq:opt} has no zero coordinates in the canonical basis, then the Frank-Wolfe algorithm with exact line-search or short step size converges linearly.
\end{corollary}
\begin{proof}[Proof of Corollary \ref{cor:accelerated_cv_FW}]
We use Lemma \ref{lem:local_strong_convexity_lq} with \citep[Theorem 2.5]{kerdreux2020uc}.
\end{proof}

\subsection{Uniform Smoothness, Rademacher type, and Generalization Bounds}\label{ssec:kakade}
Here, we show an example where the equivalence between the uniform convexity of the gauge and the Banach space's uniform convexity provides another perspective on a generalization bound for low-norm linear predictors \citep[Theorem 1]{kakade2009complexity} with strongly convex norm balls. We also generalize it to uniformly convex regularizing balls.

Consider a hypothesis class $\mathcal{F}$ of functions $f:\mathcal{X}\rightarrow\mathbb{R}$ and $n$ points $(x_i)\in\mathcal{X}\subset\mathbb{R}^m$, sampled from a distribution $\mu$ on $\mathcal{X}$.
For $(\epsilon_i)$ a sequence of i.i.d.~Bernouilli random variable, the Rademacher constant is defined as 
\begin{equation}\label{eq:rademacher_constant}\tag{Rademacher constant}
R_n(\mathcal{F}) \triangleq \mathbb{E}_{(\epsilon_i),(x_i)}\Big[\underset{f\in\mathcal{F}}{\text{sup }}\Big|\frac{1}{n} \sum_{i=1}^{n}{f(x_i)\epsilon_i} \Big|\Big].
\end{equation}
This Rademacher constant is a measure of the hypothesis class complexity, and a key quantity appearing in bounds on generalization error \citep{koltchinskii2001rademacher,bartlett2002model,bartlett2002rademacher,bartlett2005local}. In Theorem \ref{th:uniform_convexity_generalization}, we obtain upper bounds on the Rademacher constants of low-norm linear predictors in finite-dimensional spaces.
Such hypothesis classes are of the form
$\mathcal{F}_{\mathcal{C}} = \big\{f:x\in\mathcal{X}\rightarrow\langle x; w \rangle ~|~\|w\|_{\mathcal{C}}\leq 1\big\}$, where $\mathcal{C}$ is a compact convex centrally symmetric set with non-empty interior.

Besides uniform convexity or smoothness, various properties have been designed to further classify Banach spaces. For instance, the definitions \citep[Definition 5.8.]{deville1993smoothness} of Rademacher space of type $q\in[1,2]$ or cotype $p\in[2;+\infty[$ involve quantities very similar to the \ref{eq:rademacher_constant}.
Note that Rademacher of type $q$ and cotype are dual properties \citep[Proposition 1.e.17]{lindenstrauss2013classical}.

\begin{definition}[Space of Rademacher type and cotype]
A space $(\mathbb{R}^m, \|\cdot\|)$ is Rademacher of type $q\in [1,2]$ if for each finite sequence $(\epsilon_i)_{i=1}^n$ of i.i.d. Bernouilli variable and
any fixed finite sequence $(f_i)$ of elements of  $\mathbb{R}^m$, it holds that
\begin{equation}\label{eq:prop_rademacher_type_t}\tag{type $q$}
\mathbb{E}_{(\epsilon_i)} \Big(\big\| \sum_{i=1}^n \epsilon_i f_i \big\|^q\Big) \leq C \cdot \sum_{i=1}^n \|f_i\|^q.
\end{equation}
It is of cotype $q\in[2,+\infty[$ if there exists $C>0$ such that
\begin{equation}\label{eq:prop_rademacher_cotype_t}\tag{cotype $p$}
\sum_{i=1}^{n} \|f_i\|^p \leq C\cdot \mathbb{E}_{(\epsilon_i)}\Big(\Big\| \sum_{i=1}^{n}{\epsilon_i f_i}\Big\|^p\Big).
\end{equation}
\end{definition}

The Rademacher type of Banach spaces was leveraged in a variety of results in machine learning. 
For instance, for some class of low norm linear predictors, \cite{liu2017algorithmic} connect the duality between type and cotype (of the norm defining the hypothesis class) to the duality between stable (as they define it) learning and generalization bounds of the corresponding problem.

Slightly generalizing the Rademacher type, the martingale type/cotype of Banach spaces have been extensively studied in online learning. 
A series of works have shown the equivalence between optimal regret bounds and the martingale type of the space associated to the decision set \citep{srebro2011universality,rakhlin2017equivalence}.
Such links are not surprising as connections between martingale properties, the study of Banach spaces and concentration inequalities have long been known \citep{pisier1975martingales,pinelis1994optimum}, see \citep{pisier2011martingales,boucheron2013concentration} for recent references.

Uniform convexity is often invoked along with the martingale/Rademacher type property \citep[Section 6]{srebro2011universality}. 
Indeed, a uniformly smooth space of type $q\in]1,2]$ is also a Rademacher Banach space of type $q$ \citep[Lemma 5.9.]{deville1993smoothness}, while the converse is not true \citep{james1978nonreflexive}.
We recall a self-contained proof of that result \citep[Theorem 1.e.16]{lindenstrauss2013classical} for finite-dimensional spaces. 

\begin{proposition}[Uniformly Smooth and Rademacher Spaces]\label{prop:uniform_smoothness_rademacher}
Let $q\in]1,2]$. 
A normed space $(\mathbb{R}^m, \|\cdot\|)$ that is $(\alpha, q)$-uniformly smooth is also Rademacher of type $q$.
\end{proposition}
\begin{proof}[Proof of Proposition \ref{prop:uniform_smoothness_rademacher}]
Let $p\geq 2$ s.t. $1/p + 1/q = 1$.
Assume that $(\mathbb{R}^m,\|\cdot\|)$ is $(\alpha, q)$-uniformly smooth with $\alpha>0$ and $q\in]1,2]$.
Then, with Proposition \ref{prop:Duality_Polar_UC_US} \ref{itm:US_Polar_UC}, we have that $(\mathbb{R}^m, \|\cdot\|_\star)$ is $(1/(2p(2\alpha q))^{1/(q-1)},p)$-uniformly convex.
From equivalence between \ref{itm:modulus_UC} and \ref{itm:global_support_smoothness} in Theorem \ref{th:UC_set_and_support}, we finally have that $\|\cdot\|^q$ is $(c^\prime,q)$-uniformly smooth w.r.t. $\|\cdot\|$ (where $c^\prime$ only depends only on $(p,\alpha)$).
By the first-order definition of the uniform smoothness of $\|\cdot\|^q$, we have for any $h\in\mathbb{R}^m$
\begin{equation*}
   \left\{
    \begin{split}
    \|x+h\|^q \leq \|x\|^q + \langle \nabla \|\cdot\|^q(x); h\rangle + \frac{c^\prime}{q}\|h\|^q \\
    \|x-h\|^q \leq \|x\|^q - \langle \nabla \|\cdot\|^q(x); h\rangle + \frac{c^\prime}{q}\|h\|^q.
    \end{split}
   \right.
\end{equation*}
Summing these, we obtain for any $(x,h)\in\mathbb{R}^m$
\begin{equation*}
\|x+h\|^q + \|x-h\|^q - 2 \|x\|^q \leq \frac{2c^\prime}{q}\|h\|^q.
\end{equation*}
We now repeat the very same inductive argument as in \citep[Lemma 5.9.]{deville1993smoothness} and prove for any $n\geq 1$, any finite sequence of i.i.d. Bernoulli random variables $(\epsilon_i)$ and elements $(x_i)$ of $\mathbb{R}^m$ of size $n$ that
\begin{equation}\label{eq:inductive_property}
\mathbb{E}_{(\epsilon_i)}\Big( \big\| \sum_{i=1}^n{\epsilon_i x_i}\big\|^q \Big) \leq \frac{c^\prime}{q} \sum_{i=1}^{n} \|x_i\|^q.
\end{equation}
It is trivial for $n=1$. Assume \eqref{eq:inductive_property} is true for $n > 1$. We have
\begin{eqnarray*}
\mathbb{E}_{(\epsilon_i)}\Big(\big\| \sum_{i=1}^{n+1}{\epsilon_i x_i}\big\|^q\Big) &=& \frac{1}{2} \mathbb{E}_{(\epsilon_i)}\Big(\big\| \sum_{i=1}^{n}{\epsilon_i x_i} + x_{n+1}\big\|^q + \big\| \sum_{i=1}^{n}{\epsilon_i x_i} - x_{n+1}\big\|^q\Big)\\
 &\leq& \frac{1}{2} \mathbb{E}_{(\epsilon_i)}\Big( 2 \big\|\sum_{i=1}^{n}{\epsilon_i x_i}\big\|^q + 2\frac{c^\prime}{q} \|x_{n+1}\|^q \Big)\\
 &\leq & \mathbb{E}_{(\epsilon_i)}\Big( \big\|\sum_{i=1}^{n}{\epsilon_i x_i}\big\|^q \Big) + \frac{c^\prime}{q}  \|x_{n+1}\|^q \leq \frac{c^\prime}{q}  \sum_{i=1}^{n+1}{\|x_i\|^q}.
\end{eqnarray*}
Hence, $(\mathbb{R}^m, \|\cdot\|)$ is Rademacher of power type $(c^\prime/q, q)$.
\end{proof}

To the best of our knowledge, \cite{donahue1997rates} first points out the link between uniform convexity and the Rademacher type of the space in a learning framework. While the Rademacher type (resp. cotype) property is weaker than uniform smoothness (resp. convexity), establishing generalization results with uniform convexity/smoothness properties, as in \citep[Theorem 1]{kakade2009complexity} makes the assumptions much easier to interpret. This seems not to have been exploited directly to obtain upper bounds on Rademacher constants. We now extend the results of \citep[Theorem 1]{kakade2009complexity} using that insight.

\begin{theorem}\label{th:uniform_convexity_generalization}
Let $p\geq 2$ and $q\in]1,2]$ s.t. $\frac{1}{p} + \frac{1}{q}=1$.
Consider $\mathcal{F}_{\mathcal{C}} = \big\{f:x\in\mathcal{X}\rightarrow\langle x; w \rangle ~|~\|w\|_{\mathcal{C}}\leq 1\big\}$, where $\mathcal{C}$ is a centrally symmetric compact convex set with non-empty interior. 
Assume $\mathcal{C}$ is $(\alpha, p)$-uniformly convex with $p\geq 2$ and $\alpha>0$.
Then, there exists $C>0$ (a function of $p$ and $\alpha$) s.t. we have
\[
R_n(\mathcal{F}) \leq  \frac{C^{1/q} D}{n^{1/p}},
\]
where $D = \text{sup}_{x\in\mathcal{X}} \|x\|_{\mathcal{C}^\circ}$.
\end{theorem}
\begin{proof}[Proof of Theorem \ref{th:uniform_convexity_generalization}]
Since $\mathcal{C}$ is $(\alpha,p)$-uniformly convex of type $p$, the space normed with  $\|\cdot\|_{\mathcal{C}^\circ}$ is $(1/\big(2q(2\alpha p)^{q-1}\big), q)$-uniformly smooth, see Proposition \ref{prop:Duality_Polar_UC_US} \ref{itm:UC_Polar_US}. Hence with Proposition \ref{prop:uniform_smoothness_rademacher}, there exists $C>0$ (a function of $(\alpha,p)$) s.t. for any sequences $(x_i)$ and $(\epsilon_i)$ of size $n$, we have
\begin{equation}\label{eq:expectation_c_polar}
\mathbb{E}_{(\epsilon_i)}\Big(\Big\| \sum_{i}{\epsilon_i x_i}\Big\|^{q}_{\mathcal{C}^\circ}\Big) \leq C \sum \|x_i\|^{q}_{\mathcal{C}^\circ}.
\end{equation}
Then, recall that the Rademacher constant is defined as
\[
R_n(\mathcal{F}_{\mathcal{C}}) = \mathbb{E}_{(\epsilon_i),(x_i)} \Big[ \underset{f\in\mathcal{F}_{\mathcal{C}}}{\text{sup }} \frac{1}{n} \sum_{i=1}^{n}{f(x_i)\epsilon_i}\Big] = \mathbb{E}_{(\epsilon_i),(x_i)} \Big[\underset{\|w\|_{\mathcal{C}}\leq 1}{\text{sup }} \langle w ; \frac{1}{n} \sum_{i=1}^n{x_i \epsilon_i} \rangle \Big].
\]
By definition of the dual norm, we have $\langle w ; \frac{1}{n} \sum_{i=1}^n{x_i \epsilon_i} \rangle \leq \|w\|_{\mathcal{C}} \big\|\frac{1}{n} \sum_{i=1}^n{x_i \epsilon_i}\big\|_{\mathcal{C}^\circ}$, hence
\begin{equation*}
\mathbb{E}_{(\epsilon_i),(x_i)} \Big[\underset{\|w\|_{\mathcal{C}}\leq 1}{\text{sup }}\langle w ; \frac{1}{n} \sum_{i=1}^n{x_i\epsilon_i} \rangle \Big] \leq \mathbb{E}_{(\epsilon_i),(x_i)}\Big[\Big\|\frac{1}{n} \sum_{i=1}^n{x_i \epsilon_i}\Big\|_{\mathcal{C}^\circ}\Big].
\end{equation*}
Write $\theta = \Big\|\frac{1}{n} \sum_{i=1}^n{x_i \epsilon_i}\Big\|_{\mathcal{C}^\circ}$.
With $q\in]1,2]$, the function $|x|^{1/q}$ is concave on $\mathbb{R}^+$ and $\theta$ a non-negative random variable. Hence, we have $\mathbb{E}_{\epsilon} \Big[ \big(\theta^{q}\big)^{1/q}\Big] \leq \Big[\mathbb{E}_{\epsilon} (\theta^{q})\Big]^{1/q}$.
This implies that
\begin{eqnarray*}
\mathbb{E}_{(\epsilon_i)}\Big[\underset{\|w\|_{\mathcal{C}}\leq 1}{\text{sup }}\langle w ; \frac{1}{n} \sum_{i=1}^n{x_i\epsilon_i} \rangle \Big] &\leq& \frac{1}{n}\Big[\mathbb{E}_{(\epsilon_i)}\Big(\Big\| \sum_{i=1}^n{x_i \epsilon_i}\Big\|_{\mathcal{C}^\circ}^{q}\Big)\Big]^{1/q}.
\end{eqnarray*}
Hence with \eqref{eq:expectation_c_polar}, and taking the expectation w.r.t. the data points, we have
\begin{eqnarray*}
R_n(\mathcal{F}) &\leq & \frac{1}{n}\mathbb{E}_{(x_i)}\Big[C\sum_{i=1}^{n} \|x_i\|_{\mathcal{C}^\circ}^{q}\Big]^{1/q}\\
R_n(\mathcal{F}) &\leq & \frac{n^{1/q}C^{1/q}D}{n} = \frac{C^{1/q} D}{n^{1/p}},
\end{eqnarray*}
where $D = \text{sup}_{x\in\mathcal{X}} \|x\|_{\mathcal{C}^\circ}$.
\end{proof}

Upper bounds on Rademacher constants then induce generalization bounds depending on assumptions on the loss functions, see, \textit{e.g.}, \citep{kakade2009complexity}. Uniform convexity is stronger than Rademacher type properties, although a major difference is that uniform convexity admits (simple) localized definitions while martingale or Rademacher type properties are inherently global assumptions. To obtain results in learning theory that depend on the local behavior of the hypothesis class around the optimal solution, current approaches study the \emph{global} properties of a neighborhood of the hypothesis class around that solution, see, \textit{e.g.}, the local Rademacher constant \citep{bartlett2005local}. An alternative approach would then be to study local properties of the hypothesis class, for instance via local uniform convexity. This is one motivation for Theorem \ref{thm:local_UC_smoothness_gauge}. \citep{awasthi2020adversarial,Awasthi2020} prove tight upper-bound on the Rademacher constant of low-norm linear predictors with $\ell_p$ with $p>1$, which are instances of uniformly convex sets.

\subsection{Primal Averaging Frank-Wolfe on Uniformly Convex Sets}\label{ssec:PAFW}
The Primal Averaging Frank-Wolfe (PAFW) method was developed in \citep[Algorithm 4]{lan2013complexity} (see Algorithm \ref{algo:PAFW}) and replaces the projection oracle with a linear optimization oracle in Nesterov's accelerated algorithm.
We show here that the theoretical analysis of \citep[Corollary 1]{lan2013complexity}, holds in practice when the constraint set $\mathcal{C}$ is uniformly convex and the norm of the gradient functions are lower bounded on $\mathcal{C}$, \textit{i.e.}, $\text{inf}_{x\in\mathcal{C}}\|\nabla f(x)\|>c>0$.
To our knowledge, this is the first Frank-Wolfe algorithm with accelerated convergence rates relative to the baseline $\mathcal{O}(1/T)$, obtained with agnostic step-sizes, \textit{e.g.}, of the form $2/(k+2)$.

\begin{algorithm}[H]
	\caption{Primal Averaging Frank-Wolfe algorithm \label{algo:PAFW} \citep[Algorithm 4]{lan2013complexity}}
	\begin{algorithmic}
		\STATE{\textbf{Input:} $x_0\in\mathcal{C}$, $y_0\triangleq x_0$ and $(\alpha_k)\in[0,1]^\mathbb{N}$.}
		\FOR{$k=1,\ldots$}
		    \STATE $z_{k-1}=\frac{k-1}{k+1} y_{k-1} + \frac{2}{k+1} x_{k-1}.$
		    \STATE $x_k\in\text{argmax}_{v\in\mathcal{C}}\langle -\nabla f(z_{k-1}); v \rangle.$
		    \STATE $y_k = (1-\alpha_k) y_{k-1} + \alpha_k x_k$.
		\ENDFOR
	\end{algorithmic}
\end{algorithm}

\cite[Corollary 1]{lan2013complexity} yields an accelerated convergence rate of $\mathcal{O}(1/T^2)$ when some assumption is verified for the LMO. The following lemma shows that a property, similar to their assumption, holds for the LMO when the set $\mathcal{C}$ is uniformly convex. In Proposition \ref{prop:PAFW_rates}, we show how this implies new convergence rates for Primal Averaging Frank-Wolfe algorithm. This is a direct consequence of Theorem \ref{th:UC_set_and_support} \ref{itm:scaling_inequality}.
In the particular case where the set is strongly convex, this is a variation of \citep[(i) of Theorem 2.1.]{goncharov2017strong}.

\begin{lemma}\label{lem:assumption_verified}
Consider $\mathcal{C}$ a compact convex set in $\mathbb{R}^m$, $p\geq 2$, $\alpha>0$ and $(d_1,d_2)\in\mathbb{R}^m\setminus\{0\}$. 
Let $(v_1,v_2)\in\partial\mathcal{C}$ s.t. $d_i\in N_{\mathcal{C}}(v_i)$ for $i=1,2$.
If $\mathcal{C}$ is $(\alpha, p)$-uniformly convex, then we have
\[
\|v_1 - v_2\| \leq \frac{1}{\big[2\alpha\big(\|d_1\|_\star + \|d_2\|_\star\big)\big]^{1/(p-1)}} \|d_1 - d_2\|_\star^{1/(p-1)}.
\]
\end{lemma}
\begin{proof}[Proof of Lemma \ref{lem:assumption_verified}]
Because $\mathcal{C}$ is $(\alpha, p)$-uniformly convex, via \ref{itm:scaling_inequality} of Theorem \ref{th:UC_set_and_support} applied to $(v_i,d_i)$ for $i=1,2$, we obtain $\langle d_1 ; v_1 - v_2 \rangle \geq 2\alpha \|d_1\|_\star \|v_1 - v_2\|^p$ and $\langle d_2 ; v_2 - v_1 \rangle \geq 2\alpha \|d_2\|_\star \|v_2 - v_1\|^p$.
Summing the two inequalities implies that $\langle d_1 -d_2; v_1 - v_2 \rangle \geq 2\alpha \big(\|d_1\|_\star + \|d_2\|_\star \big) \|v_1 - v_2\|^p$.
Finally with Cauchy-Schwartz, we obtain $\|v_1 - v_2\| \leq \frac{1}{\big[2\alpha\big(\|d_1\|_\star + \|d_2\|_\star\big)\big]^{1/(p-1)}} \|d_1 - d_2\|_\star^{1/(p-1)}$.
\end{proof}
Hence, if the norms of the $d_i$ for $i=1,2$ are lower bounded by $c>0$, and the set is $(\alpha, p)$-uniformly convex with $p\in[2,3]$, we obtain that the condition described in \citep{lan2013complexity} is valid and of the form 
\[
\|v_1 - v_2\| \leq 1/(2 \alpha c)^{1/(p-1)}\|d_1 - d_2\|^{1/(p-1)}.
\]
When $\mathcal{C}$ is strongly convex and $\text{inf}_{x\in\mathcal{C}}\|\nabla f(x)\|>c$, it is already known that vanilla Frank-Wolfe with short steps or exact line-search converges linearly \citep{demyanov1970,dunn1979rates}.
The difference is PAFW has accelerated convergence results with agnostic step sizes, \textit{i.e.}, $\alpha_k=\frac{2}{k+2}$, which is much cheaper to implement and also do not require knowledge of $L$ in $f$.
When the set is uniformly convex but not strongly convex, \cite{kerdreux2020uc} obtain sublinear rates for vanilla Frank-Wolfe algorithms on uniformly convex set with short steps or exact line-search. The rates in Proposition \ref{prop:PAFW_rates} are strictly inferior to the $\mathcal{O}(1/T^{1/(1-2/p)})$ in \citep{kerdreux2020uc} obtained with the same structural assumptions. However, to the best of our knowledge, the accelerated convergence rates of Algorithm \ref{algo:PAFW} are the only accelerated convergence rates holding with oblivious step-sizes.

\begin{proposition}\label{prop:PAFW_rates}
Consider $f$ a convex $L$-smooth function w.r.t. $\|\cdot\|$ and $p\geq 2$, $\alpha>0$ .
Assume $\mathcal{C}$ is $(\alpha, p)$-uniformly convex and $\text{inf}_{x\in\mathcal{C}}\|\nabla f(x)\|>c>0$. 
Then the iterates $(y_k)$ of PAFW (Algorithm \ref{algo:PAFW}) with $\alpha_k=\frac{2}{k+2}$ satisfy
\begin{equation*}
  f(y_k) - f^* \leq 2L\Big(\frac{6L D_{\|\cdot\|}}{4\alpha c}\Big)^{1/(p-1)}\left\{
    \begin{split}
    &\frac{1}{k^{(p+1)/(p-1)}}&\text{when } p\in]3,+\infty[\\
    &\frac{\log(k+1)}{k^2}&\text{when } p=3\\
    &\frac{3-p}{p-1}\frac{1}{k^2}&\text{when } p\in[2;3[.
    \end{split}
   \right.
\end{equation*}
where $D_{\|\cdot\|}$ is the diameter of $\mathcal{C}$ w.r.t. $\|\cdot\|$.
\end{proposition}
\begin{proof}[Proof of Proposition \ref{prop:PAFW_rates}]
From \citep[Theorem 7]{lan2013complexity}, we have
\[
f(y_k) - f^* \leq \frac{2L}{k(k+1)}\sum_{i=1}^{k}{\|x_i - x_{i-1}\|^2}.
\]
Then, from Lemma \ref{lem:assumption_verified}, since in Algorithm \ref{algo:PAFW}, $x_i$ are such that $x_i\in\text{argmax}_{x\in\mathcal{C}}\langle \nabla f(z_{k-1});x \rangle$, we have 
\[
\|x_{i} - x_{i-1}\| \leq \frac{1}{\big[2\alpha \big(\|\nabla f(z_{i-1})\|_\star + \|\nabla f(z_{i-2})\|_\star\big)\big]^{1/(p-1)}} \|\nabla f(z_{i-1}) - \nabla f(z_{i-2})\|_\star^{1/(p-1)}.
\]
Then, since $z_i\in\mathcal{C}$ and $\|\nabla f(z_{i-1}) - \nabla f(z_{i-2}) \|_\star \leq \frac{6L D_{\|\cdot\|}}{i+1}$ (see \citep[(4.3)]{lan2013complexity}, we have 
\[
\|x_i - x_{i-1}\| \leq \frac{\big(6LD_{ \|\cdot\|}\big)^{1/(p-1)}}{(4\alpha c)^{1/(p-1)}}\frac{1}{(i+1)^{1/(p-1)}}.
\]
Simple computations \citep{lan2013complexity} imply that
\begin{equation*}
  \sum_{i=1}^{k}{\frac{1}{i^{2/(p-1)}}} = \left\{
    \begin{split}
    &(k+1)^{\frac{p-3}{p-1}}~~\text{when } p\in[3,+\infty[\\
    &\text{log}(k+1)~~\text{when } p=3\\
    &\frac{3-p}{p-1}~~\text{when } p\in[2;3[.
    \end{split}
   \right.
\end{equation*}
Hence,
\begin{equation*}
  f(y_k) - f^* \leq 2L\Big(\frac{6L D_{\|\cdot\|}}{4\alpha c}\Big)^{1/(p-1)}\left\{
    \begin{split}
    &\frac{1}{k^{(p+1)/(p-1)}}~~\text{when } p\in[3,+\infty[\\
    &\frac{\text{log}(k+1)}{k^2}~~\text{when } p=3\\
    &\frac{3-p}{p-1}\frac{1}{k^2}~~\text{when } p\in[2;3[.
    \end{split}
   \right.
\end{equation*}
\end{proof}


\newpage 

\subsection*{Acknowledgment}
TK is very much indebted to Pierre-Cyril Aubin for the many discussions around uniform convexity in a learning framework. Research reported in this paper was partially supported through the Research Campus Modal funded by the German Federal Ministry of Education and Research (fund numbers 05M14ZAM,05M20ZBM) as well as the Deutsche Forschungsgemeinschaft (DFG) through the DFG Cluster of Excellence MATH+. AA is at the d\'epartement d'informatique de l'\'Ecole Normale Sup\'erieure, UMR CNRS 8548, PSL Research University, 75005 Paris, France, and INRIA. AA would like to acknowledge support from the {\em ML and Optimisation} joint research initiative with the {\em fonds AXA pour la recherche} and Kamet Ventures, a Google focused award, as well as funding by the French government under management of Agence Nationale de la Recherche as part of the "Investissements d'avenir" program, reference ANR-19-P3IA-0001 (PRAIRIE 3IA Institute).

{\small \bibliographystyle{alpha}
\bibsep 1ex
\bibliography{biblio_set_HEB,MainPerso.bib}}

\newcommand{\etalchar}[1]{$^{#1}$}
\begin{thebibliography}{BFTGT19}

\bibitem[ABRS10]{Attouch10}
H{\'e}dy Attouch, J{\'e}r{\^o}me Bolte, Patrick Redont, and Antoine Soubeyran.
\newblock Proximal alternating minimization and projection methods for
  nonconvex problems: An approach based on the {K}urdyka-{L}ojasiewicz
  inequality.
\newblock {\em Mathematics of Operations Research}, 35(2):438--457, 2010.

\bibitem[AFM20a]{awasthi2020adversarial}
Pranjal Awasthi, Natalie Frank, and Mehryar Mohri.
\newblock Adversarial learning guarantees for linear hypotheses and neural
  networks.
\newblock In {\em International Conference on Machine Learning}, pages
  431--441. PMLR, 2020.

\bibitem[AFM20b]{Awasthi2020}
Pranjal Awasthi, Natalie Frank, and Mehryar Mohri.
\newblock On the {R}ademacher complexity of linear hypothesis sets.
\newblock {\em arXiv:2007.11045}, 2020.

\bibitem[ALLW18]{abernethy2018faster}
Jacob Abernethy, Kevin Lai, Kfir Levy, and Jun-Kun Wang.
\newblock Faster rates for convex-concave games.
\newblock In {\em Conference On Learning Theory}, pages 1595--1625. PMLR, 2018.

\bibitem[ALW19]{abernethy2019last}
Jacob Abernethy, Kevin Lai, and Andre Wibisono.
\newblock Last-iterate convergence rates for min-max optimization.
\newblock {\em arXiv preprint arXiv:1906.02027}, 2019.

\bibitem[AP95]{aze1995uniformly}
Dominique Az{\'e} and Jean-Paul Penot.
\newblock Uniformly convex and uniformly smooth convex functions.
\newblock In {\em Annales de la Facult{\'e} des sciences de Toulouse:
  Math{\'e}matiques}, volume~4, pages 705--730, 1995.

\bibitem[AR09]{abernethy2009beating}
Jacob Abernethy and Alexander Rakhlin.
\newblock Beating the adaptive bandit with high probability.
\newblock In {\em 2009 Information Theory and Applications Workshop}, pages
  280--289. IEEE, 2009.

\bibitem[Asp68]{asplund1968frechet}
Edgar Asplund.
\newblock Fr{\'e}chet differentiability of convex functions.
\newblock {\em Acta Mathematica}, 121(1):31--47, 1968.

\bibitem[AYAS09]{abbasi2009forced}
Yasin Abbasi-Yadkori, Andr{\'a}s Antos, and Csaba Szepesv{\'a}ri.
\newblock Forced-exploration based algorithms for playing in stochastic linear
  bandits.
\newblock Citeseer, 2009.

\bibitem[Bac20]{bach2020effectiveness}
Francis Bach.
\newblock On the effectiveness of {R}ichardson extrapolation in machine
  learning.
\newblock {\em arXiv preprint arXiv:2002.02835}, 2020.

\bibitem[BBL02]{bartlett2002model}
Peter Bartlett, St{\'e}phane Boucheron, and G{\'a}bor Lugosi.
\newblock Model selection and error estimation.
\newblock {\em Machine Learning}, 48(1-3):85--113, 2002.

\bibitem[BBM05]{bartlett2005local}
Peter~L Bartlett, Olivier Bousquet, and Shahar Mendelson.
\newblock Local {R}ademacher complexities.
\newblock {\em The Annals of Statistics}, 33(4):1497--1537, 2005.

\bibitem[BCKP20a]{bhaskara2020online}
Aditya Bhaskara, Ashok Cutkosky, Ravi Kumar, and Manish Purohit.
\newblock Online learning with imperfect hints.
\newblock In {\em International Conference on Machine Learning}, pages
  822--831. PMLR, 2020.

\bibitem[BCKP20b]{bhaskara2020onlineMany}
Aditya Bhaskara, Ashok Cutkosky, Ravi Kumar, and Manish Purohit.
\newblock Online linear optimization with many hints.
\newblock {\em arXiv:2010.03082}, 2020.

\bibitem[BCL18]{bubeck2018sparsity}
S{\'e}bastien Bubeck, Michael Cohen, and Yuanzhi Li.
\newblock Sparsity, variance and curvature in multi-armed bandits.
\newblock In {\em Algorithmic Learning Theory}, pages 111--127. PMLR, 2018.

\bibitem[BDL07]{bolte2007lojasiewicz}
J{\'e}r{\^o}me Bolte, Aris Daniilidis, and Adrian Lewis.
\newblock The {L}ojasiewicz inequality for nonsmooth subanalytic functions with
  applications to subgradient dynamical systems.
\newblock {\em SIAM Journal on Optimization}, 17(4):1205--1223, 2007.

\bibitem[BDLM10]{Bolte10}
J\'er\^ome Bolte, Aris Daniilidis, Olivier Ley, and Laurent Mazet.
\newblock Characterizations of {L}ojasiewicz inequalities: Subgradient flows,
  talweg, convexity.
\newblock {\em Transactions of the American Mathematical Society},
  362(6):3319--3363, 2010.

\bibitem[Bea11]{beauzamy2011introduction}
Bernard Beauzamy.
\newblock {\em Introduction to Banach spaces and their geometry}.
\newblock Elsevier, 2011.

\bibitem[BFTGT19]{bassily2019private}
Raef Bassily, Vitaly Feldman, Kunal Talwar, and Abhradeep Guha~Thakurta.
\newblock Private stochastic convex optimization with optimal rates.
\newblock {\em Advances in Neural Information Processing Systems},
  32:11282--11291, 2019.

\bibitem[BGHV09]{borwein2009uniformly}
J.~Borwein, A.~Guirao, Petr. H{\'a}jek, and J.~Vanderwerff.
\newblock Uniformly convex functions on {B}anach spaces.
\newblock {\em Proceedings of the American Mathematical Society},
  137(3):1081--1091, 2009.

\bibitem[BLM13]{boucheron2013concentration}
St{\'e}phane Boucheron, G{\'a}bor Lugosi, and Pascal Massart.
\newblock {\em Concentration inequalities: A nonasymptotic theory of
  independence}.
\newblock Oxford university press, 2013.

\bibitem[BM02]{bartlett2002rademacher}
Peter Bartlett and Shahar Mendelson.
\newblock {R}ademacher and {G}aussian complexities: Risk bounds and structural
  results.
\newblock {\em Journal of Machine Learning Research}, 3(Nov):463--482, 2002.

\bibitem[BNPS17]{bolte2017error}
J{\'e}r{\^o}me Bolte, Trong~Phong Nguyen, Juan Peypouquet, and Bruce Suter.
\newblock From error bounds to the complexity of first-order descent methods
  for convex functions.
\newblock {\em Mathematical Programming}, 165(2):471--507, 2017.

\bibitem[Cla36]{clarkson1936uniformly}
James Clarkson.
\newblock Uniformly convex spaces.
\newblock {\em Transactions of the American Mathematical Society},
  40(3):396--414, 1936.

\bibitem[CLK19]{chen2019renyi}
Chen Chen, Jaewoo Lee, and Dan Kifer.
\newblock Renyi differentially private {ERM} for smooth objectives.
\newblock In {\em The 22nd International Conference on Artificial Intelligence
  and Statistics}, pages 2037--2046, 2019.

\bibitem[CP19]{combettes2019revisiting}
Cyrille Combettes and Sebastian Pokutta.
\newblock Revisiting the approximate {C}arath\'eodory problem via the
  {F}rank-{W}olfe algorithm.
\newblock {\em arXiv preprint arXiv:1911.04415}, 2019.

\bibitem[DDGS97]{donahue1997rates}
Michael Donahue, Christian Darken, Leonid Gurvits, and Eduardo Sontag.
\newblock Rates of convex approximation in non-{H}ilbert spaces.
\newblock {\em Constructive Approximation}, 13(2):187--220, 1997.

\bibitem[DFHJ17]{dekel2017online}
Ofer Dekel, Arthur Flajolet, Nika Haghtalab, and Patrick Jaillet.
\newblock Online learning with a hint.
\newblock In {\em Advances in Neural Information Processing Systems}, pages
  5299--5308, 2017.

\bibitem[dGJ18]{d2018optimal}
Alexandre d'Aspremont, Cristobal Guzman, and Martin Jaggi.
\newblock Optimal affine-invariant smooth minimization algorithms.
\newblock {\em SIAM Journal on Optimization}, 28(3):2384--2405, 2018.

\bibitem[DGZ93]{deville1993smoothness}
Robert Deville, Gilles Godefroy, and V{\'a}clav Zizler.
\newblock {\em Smoothness and renormings in Banach spaces}.
\newblock Longman Scientific Technical, Harlow, 1993.

\bibitem[DH19]{du2019linear}
Simon Du and Wei Hu.
\newblock Linear convergence of the primal-dual gradient method for
  convex-concave saddle point problems without strong convexity.
\newblock In {\em The 22nd International Conference on Artificial Intelligence
  and Statistics}, pages 196--205. PMLR, 2019.

\bibitem[DR70]{demyanov1970}
V.~F. Demyanov and A.~M. Rubinov.
\newblock Approximate methods in optimization problems.
\newblock {\em Modern Analytic and Computational Methods in Science and
  Mathematics}, 1970.

\bibitem[Dun79]{dunn1979rates}
Joseph Dunn.
\newblock Rates of convergence for conditional gradient algorithms near
  singular and nonsingular extremals.
\newblock {\em SIAM Journal on Control and Optimization}, 17(2):187--211, 1979.

\bibitem[EBEGT19]{el2019generalization}
Othman El~Balghiti, Adam Elmachtoub, Paul Grigas, and Ambuj Tewari.
\newblock Generalization bounds in the predict-then-optimize framework.
\newblock In {\em Advances in Neural Information Processing Systems}, pages
  14412--14421, 2019.

\bibitem[FKT20]{feldman2020private}
Vitaly Feldman, Tomer Koren, and Kunal Talwar.
\newblock Private stochastic convex optimization: Optimal rates in linear time.
\newblock In {\em Proceedings of the 52nd Annual ACM SIGACT Symposium on Theory
  of Computing}, pages 439--449, 2020.

\bibitem[GH15]{garber2015faster}
Dan Garber and Elad Hazan.
\newblock Faster rates for the {F}rank-{W}olfe method over strongly-convex
  sets.
\newblock In {\em 32nd International Conference on Machine Learning, ICML
  2015}, 2015.

\bibitem[GI17]{goncharov2017strong}
Vladimir Goncharov and Grigorii Ivanov.
\newblock Strong and weak convexity of closed sets in a {H}ilbert space.
\newblock In {\em Operations research, engineering, and cyber security}, pages
  259--297. Springer, 2017.

\bibitem[GJLJ17]{gidel2017frank}
Gauthier Gidel, Tony Jebara, and Simon Lacoste-Julien.
\newblock {F}rank-{W}olfe algorithms for saddle point problems.
\newblock In {\em Artificial Intelligence and Statistics}, pages 362--371.
  PMLR, 2017.

\bibitem[HLGS16]{huang2016following}
Ruitong Huang, Tor Lattimore, Andr{\'a}s Gy{\"o}rgy, and Csaba Szepesv{\'a}ri.
\newblock Following the leader and fast rates in linear prediction: Curved
  constraint sets and other regularities.
\newblock In {\em Advances in Neural Information Processing Systems}, pages
  4970--4978, 2016.

\bibitem[HLGS17]{huang2017following}
Ruitong Huang, Tor Lattimore, Andr{\'a}s Gy{\"o}rgy, and Csaba Szepesv{\'a}ri.
\newblock Following the leader and fast rates in online linear prediction:
  Curved constraint sets and other regularities.
\newblock {\em The Journal of Machine Learning Research}, 18(1):5325--5355,
  2017.

\bibitem[IN14]{iouditski2014primal}
Anatoli Iouditski and Yuri Nesterov.
\newblock Primal-dual subgradient methods for minimizing uniformly convex
  functions.
\newblock {\em arXiv preprint arXiv:1401.1792}, 2014.

\bibitem[INS{\etalchar{+}}19]{iyengar2019towards}
Roger Iyengar, Joseph Near, Dawn Song, Om~Thakkar, Abhradeep Thakurta, and Lun
  Wang.
\newblock Towards practical differentially private convex optimization.
\newblock In {\em 2019 IEEE Symposium on Security and Privacy (SP)}, pages
  299--316. IEEE, 2019.

\bibitem[Jag13]{jaggi2013revisiting}
Martin Jaggi.
\newblock Revisiting {F}rank-{W}olfe: Projection-free sparse convex
  optimization.
\newblock In {\em Proceedings of the 30th international conference on machine
  learning}, 2013.

\bibitem[Jam78]{james1978nonreflexive}
RC~James.
\newblock Nonreflexive spaces of type 2.
\newblock {\em Israel Journal of Mathematics}, 30(1-2):1--13, 1978.

\bibitem[JST{\etalchar{+}}14]{jaggi2014communication}
Martin Jaggi, Virginia Smith, Martin Tak{\'a}c, Jonathan Terhorst, Sanjay
  Krishnan, Thomas Hofmann, and Michael Jordan.
\newblock Communication-efficient distributed dual coordinate ascent.
\newblock {\em Advances in neural information processing systems},
  27:3068--3076, 2014.

\bibitem[KBGY20]{kuru2020differentially}
Nurdan Kuru, {\.I}lker Birbil, Mert Gurbuzbalaban, and Sinan Yildirim.
\newblock Differentially private accelerated optimization algorithms.
\newblock {\em arXiv preprint arXiv:2008.01989}, 2020.

\bibitem[KCd17]{kerdreux2017approximate}
Thomas Kerdreux, Igor Colin, and Alexandre d'Aspremont.
\newblock An approximate {S}hapley-{F}olkman theorem.
\newblock {\em arXiv preprint arXiv:1712.08559}, 2017.

\bibitem[KdP19]{kerdreux2019restarting}
Thomas Kerdreux, Alexandre d’Aspremont, and Sebastian Pokutta.
\newblock Restarting {F}rank-{W}olfe.
\newblock In {\em The 22nd International Conference on Artificial Intelligence
  and Statistics}, pages 1275--1283. PMLR, 2019.

\bibitem[KdP20]{kerdreux2020uc}
Thomas Kerdreux, Alexandre d'Aspremont, and Sebastian Pokutta.
\newblock Projection-free optimization on uniformly convex sets.
\newblock {\em arXiv:2004.11053}, 2020.

\bibitem[Ker20]{kerdreux2020accelerating}
Thomas Kerdreux.
\newblock {\em Accelerating conditional gradient methods}.
\newblock PhD thesis, Universit{\'e} Paris sciences et lettres, 2020.

\bibitem[KLLJS20]{kerdreux2020aff}
Thomas Kerdreux, Lewis Liu, Simon Lacoste-Julien, and Damien Scieur.
\newblock Affine invariant analysis of {F}rank-{W}olfe on strongly convex sets.
\newblock {\em arXiv:2011.03351}, 2020.

\bibitem[Kol01]{koltchinskii2001rademacher}
Vladimir Koltchinskii.
\newblock {R}ademacher penalties and structural risk minimization.
\newblock {\em IEEE Transactions on Information Theory}, 47(5):1902--1914,
  2001.

\bibitem[K{\"o}t83]{kothe1983topological}
Gottfried K{\"o}the.
\newblock Topological vector spaces.
\newblock In {\em Topological Vector Spaces I}, pages 123--201. Springer, 1983.

\bibitem[KST09]{kakade2009complexity}
Sham Kakade, Karthik Sridharan, and Ambuj Tewari.
\newblock On the complexity of linear prediction: Risk bounds, margin bounds,
  and regularization.
\newblock In {\em Advances in neural information processing systems}, pages
  793--800, 2009.

\bibitem[Lan13]{lan2013complexity}
Guanghui Lan.
\newblock The complexity of large-scale convex programming under a linear
  optimization oracle.
\newblock {\em arXiv preprint arXiv:1309.5550}, 2013.

\bibitem[Lin63]{lindenstrauss1963modulus}
Joram Lindenstrauss.
\newblock On the modulus of smoothness and divergent series in banach spaces.
\newblock {\em The Michigan Mathematical Journal}, 10(3):241--252, 1963.

\bibitem[LLNT17]{liu2017algorithmic}
Tongliang Liu, G{\'a}bor Lugosi, Gergely Neu, and Dacheng Tao.
\newblock Algorithmic stability and hypothesis complexity.
\newblock {\em arXiv preprint arXiv:1702.08712}, 2017.

\bibitem[LR15]{lee2015distributed}
Ching-Pei Lee and Dan Roth.
\newblock Distributed box-constrained quadratic optimization for dual linear
  {SVM}.
\newblock In {\em International Conference on Machine Learning}, pages
  987--996, 2015.

\bibitem[LS19]{liang2019interaction}
Tengyuan Liang and James Stokes.
\newblock Interaction matters: A note on non-asymptotic local convergence of
  generative adversarial networks.
\newblock In {\em The 22nd International Conference on Artificial Intelligence
  and Statistics}, pages 907--915. PMLR, 2019.

\bibitem[LT13]{lindenstrauss2013classical}
Joram Lindenstrauss and Lior Tzafriri.
\newblock {\em Classical Banach spaces II: Function spaces}, volume~97.
\newblock Springer Science \& Business Media, 2013.

\bibitem[Mol20]{Molinaro20}
Marco Molinaro.
\newblock Curvature of feasible sets in offline and online optimization.
\newblock {\em arXiv:2002.03213}, 2020.

\bibitem[MOP20]{mokhtari2020unified}
Aryan Mokhtari, Asuman Ozdaglar, and Sarath Pattathil.
\newblock A unified analysis of extra-gradient and optimistic gradient methods
  for saddle point problems: Proximal point approach.
\newblock In {\em International Conference on Artificial Intelligence and
  Statistics}, pages 1497--1507. PMLR, 2020.

\bibitem[MSJ{\etalchar{+}}15]{ma2015adding}
Chenxin Ma, Virginia Smith, Martin Jaggi, Michael Jordan, Peter Richt{\'a}rik,
  and Martin Tak{\'a}c.
\newblock Adding vs. averaging in distributed primal-dual optimization.
\newblock In {\em International Conference on Machine Learning}, pages
  1973--1982. PMLR, 2015.

\bibitem[Nes05]{nesterov2005smooth}
Yu~Nesterov.
\newblock Smooth minimization of non-smooth functions.
\newblock {\em Mathematical programming}, 103(1):127--152, 2005.

\bibitem[Nes15]{Nest15}
Yu~Nesterov.
\newblock Universal gradient methods for convex optimization problems.
\newblock {\em Mathematical Programming}, 152(1-2):381--404, 2015.

\bibitem[Pin94]{pinelis1994optimum}
Iosif Pinelis.
\newblock Optimum bounds for the distributions of martingales in {B}anach
  spaces.
\newblock {\em The Annals of Probability}, pages 1679--1706, 1994.

\bibitem[Pis75]{pisier1975martingales}
Gilles Pisier.
\newblock Martingales with values in uniformly convex spaces.
\newblock {\em Israel Journal of Mathematics}, 20(3-4):326--350, 1975.

\bibitem[Pis11]{pisier2011martingales}
Gilles Pisier.
\newblock Martingales in {B}anach spaces (in connection with type and cotype).
  course {IHP}, {F}eb. 2--8, 2011.

\bibitem[Pol66]{polyak1966existence}
Boris Polyak.
\newblock Existence theorems and convergence of minimizing sequences in
  extremum problems with restrictions.
\newblock In {\em Soviet Math. Dokl}, volume~7, pages 72--75, 1966.

\bibitem[RBWM19]{rector2019revisiting}
Jarrid Rector-Brooks, Jun-Kun Wang, and Barzan Mozafari.
\newblock Revisiting projection-free optimization for strongly convex
  constraint sets.
\newblock In {\em Proceedings of the AAAI Conference on Artificial
  Intelligence}, volume~33, pages 1576--1583, 2019.

\bibitem[Rd20]{roulet2020sharpness}
Vincent Roulet and Alexandre d'Aspremont.
\newblock Sharpness, restart, and acceleration.
\newblock {\em SIAM Journal on Optimization}, 30(1):262--289, 2020.

\bibitem[Roc70]{rockafellar1970convex}
Tyrrell Rockafellar.
\newblock {\em Convex analysis}.
\newblock Princeton university press, 1970.

\bibitem[RS17]{rakhlin2017equivalence}
Alexander Rakhlin and Karthik Sridharan.
\newblock On equivalence of martingale tail bounds and deterministic regret
  inequalities.
\newblock In {\em Conference on Learning Theory}, pages 1704--1722. PMLR, 2017.

\bibitem[RT10]{rusmevichientong2010linearly}
Paat Rusmevichientong and John Tsitsiklis.
\newblock Linearly parameterized bandits.
\newblock {\em Mathematics of Operations Research}, 35(2):395--411, 2010.

\bibitem[Sch14]{schneider2014convex}
Rolf Schneider.
\newblock {\em Convex bodies: The Brunn--Minkowski theory}.
\newblock Cambridge university press, 2014.

\bibitem[Sch16]{Schn16}
Markus Schneider.
\newblock Probability inequalities for kernel embeddings in sampling without
  replacement.
\newblock In {\em Artificial Intelligence and Statistics}, pages 66--74, 2016.

\bibitem[SFM{\etalchar{+}}17]{smith2017cocoa}
Virginia Smith, Simone Forte, Chenxin Ma, Martin Tak{\'a}{\v{c}}, Michael
  Jordan, and Martin Jaggi.
\newblock Cocoa: A general framework for communication-efficient distributed
  optimization.
\newblock {\em The Journal of Machine Learning Research}, 18(1):8590--8638,
  2017.

\bibitem[SST11]{srebro2011universality}
Nati Srebro, Karthik Sridharan, and Ambuj Tewari.
\newblock On the universality of online mirror descent.
\newblock In {\em Advances in neural information processing systems}, pages
  2645--2653, 2011.

\bibitem[ST10]{sridharan2010convex}
Karthik Sridharan and Ambuj Tewari.
\newblock Convex games in {B}anach spaces.
\newblock In {\em Conference on Learning Theory}. Citeseer, 2010.

\bibitem[Sti18]{stich2018local}
Sebastian Stich.
\newblock Local {SGD} converges fast and communicates little.
\newblock {\em arXiv preprint arXiv:1805.09767}, 2018.

\bibitem[TTZ14]{talwar2014private}
Kunal Talwar, Abhradeep Thakurta, and Li~Zhang.
\newblock Private empirical risk minimization beyond the worst case: The effect
  of the constraint set geometry.
\newblock {\em arXiv preprint arXiv:1411.5417}, 2014.

\bibitem[VV20]{veliov2020gradient}
V.M. Veliov and Phan~Tu Vuong.
\newblock Gradient methods on strongly convex feasible sets and optimal control
  of affine systems.
\newblock {\em Applied Mathematics \& Optimization}, 81(3):1021--1054, 2020.

\bibitem[WA18]{wang2018acceleration}
Jun-Kun Wang and Jacob Abernethy.
\newblock Acceleration through optimistic no-regret dynamics.
\newblock In {\em Advances in Neural Information Processing Systems}, pages
  3824--3834, 2018.

\bibitem[ZZMW17]{zhang2017efficient}
Jiaqi Zhang, Kai Zheng, Wenlong Mou, and Liwei Wang.
\newblock Efficient private {ERM} for smooth objectives.
\newblock {\em arXiv preprint arXiv:1703.09947}, 2017.

\bibitem[Zǎ83]{zalinescu1983uniformly}
C~Zǎlinescu.
\newblock On uniformly convex functions.
\newblock {\em Journal of Mathematical Analysis and Applications},
  95(2):344--374, 1983.

\bibitem[Zǎ02]{zalinescu2002convex}
Constantin Zǎlinescu.
\newblock {\em Convex analysis in general vector spaces}.
\newblock World scientific, 2002.

\end{thebibliography}

\appendix

\end{document}